\documentclass[a4paper,12pt]{amsart}
\usepackage{yfonts,amsmath,amsfonts,amssymb,amsthm}
\usepackage{graphics}
\usepackage{epsfig}
\usepackage{esint}
\usepackage{hyperref}
\usepackage{xcolor}

\newcommand{\ignore}[1]{}

\newtheorem{definition}{Definition}

\newtheorem{theorem}{Theorem}
\newtheorem{remark}{Remark}
\newtheorem{lemma}{Lemma}

\newcommand{\z}{\mathsf{z}}
\newcommand{\n}{{\bf n}}
\newcommand{\0}{{\bf 0}}

\title[Lecture notes on tree-free regularity structures]{Lecture notes on\\ tree-free regularity structures}
\author[F.~Otto]{Felix Otto}
\author[K.~Seong]{Kihoon Seong}
\author[M.~Tempelmayr]{Markus Tempelmayr}

\renewenvironment{abstract}
{
\begin{center}
\begin{minipage}{.9\textwidth}\small\textbf{Abstract}.\noindent
}
{
\end{minipage}
\end{center}
}

\newenvironment{keyword}
{
\begin{center}
\begin{minipage}{.9\textwidth}\small\textbf{Keywords}:\noindent
}
{
\end{minipage}
\end{center}
}

\newenvironment{msc}
{
\begin{center}
\begin{minipage}{.9\textwidth}\small\textbf{MSC 2020}:\noindent
}
{
\end{minipage}
\end{center}
}

\begin{document}

\maketitle

\begin{abstract}
These lecture notes are intended as reader's digest of recent
work on a diagram-free approach to the renormalized centered model in
Hairer's regularity structures. More precisely, it is about
the stochastic estimates of the centered model, 
based on Malliavin calculus and a spectral gap assumption. We focus on
a specific parabolic partial differential equation in quasi-linear form
driven by (white) noise.

We follow a natural renormalization strategy based on preserving symmetries,
and carefully introduce Hairer's notion of a centered model, 
which provides the coefficients in a formal series expansion of a general solution.
We explain how the Malliavin derivative in conjunction with 
Hairer's re-expansion map allows to reformulate this definition
in a way that is stable under removing the small-scale regularization.

A few exemplary proofs are provided, 
both of analytic and of algebraic character. 
The working horse of the analytic arguments is an ``annealed'' Schauder estimate
and related Liouville principle, which is provided.
The algebra of formal power series, in variables that play the role
of coordinates of the solution manifold, and its algebra morphisms
are the key algebraic objects. 
\end{abstract}

\smallskip


\smallskip

\begin{msc} 
60H17, 
60L30, 
60H07, 
81T16, 
35K59.  
\end{msc}

\tableofcontents
\newpage

\noindent
The theory of regularity structures by Hairer 
provides a systematic way to treat the small-scale divergences in 
singular semi-linear stochastic PDEs.
Quintessential models of mathematical physics like the
dynamical $\phi^4_3$ model or the KPZ equation have been treated. 
Inspired by Lyon's theory of rough paths, this theory separates probabilistic 
and analytical aspects: 
\begin{itemize}
\item Centered model.
In a first probabilistic step, the coefficients of a local formal power series
representation of a general solution of the renormalized PDE 
are constructed and estimated; 
the coefficients are indexed by (decorated) trees, 
and their stochastic estimate follows the diagrammatic approach 
to renormalization of quantum field theories.\\
\item Modelled distribution. 
In a second analytical step, inspired by Gubinelli's controlled rough path,
the solution of a specific initial value problem is found as a fixed point based
on modulating and truncating the formal power series . This step is
purely deterministic.
\end{itemize}
This automated two-pronged approach relies on an understanding of the algebraic nature
of the re-expansion maps that allow to pass from one base-point to another
in the local power series representation, in form of the ``structure group''. 
The main progress of regularity structures over
the term-by-term treatment in the mathematical physics literature is that thanks to
centering and re-expansion, the second step yields a rigorous (small data) well-posedness result. 
As an introductory text to the theory of regularity structures we recommend \cite{Hai16, CW17, BCZ23}. 

\medskip

\noindent
In \cite{OSSW}, motivated by the extension to a quasi-linear setting
featuring a general non-linearity $a(u)$, 
an alternative realization of
Hairer's regularity structures was proposed; it replaces
trees with a more greedy index set. 
This index set of multi-indices naturally comes up when 
writing a general solution $u$ as a functional of $a$, 
or rather as a function of the coefficients of $a$ in its power series expansion.
In \cite{OSSW} it was established 
that any solution of the renormalized PDE can be locally
approximated by a modelled distribution. This 
a-priori estimate was obtained under the assumption 
that the natural stochastic estimates on the centered model are available. 


In \cite{LOTT} this program was continued: 
Based on scaling and other symmetries,
a canonical renormalization of the PDE and its centered model was proposed,
and the centered model was stochastically constructed and estimated.
These notes present selected aspects of \cite{LOTT}, providing additional motivation. 
For a simpler setting where no renormalization and thus only purely
deterministic estimates are 
needed, we recommend to also have a look at\footnote{however,
the setting in \cite{LO} is different in the sense that it imposes an artificial space-time
periodicity: on the one hand, this allows to separate construction from estimation,
on the other hand, it obfuscates the quintessential scaling} \cite{LO}. 
Let us mention that the stochastic estimates obtained in \cite{LOTT} are analogous to the one of \cite{ch16} in the tree-based setting, 
however the assumption on the noise and the method of proof differ radically. 
The method presented here, based on Malliavin calculus and a spectral gap assumption, 
has been picked up in \cite{HS} and implemented in a general context in the tree-based setting. 
Also \cite{BB} used similar ideas to obtain stochastic estimates for the generalized KPZ equation, by iteratively applying the spectral gap inequality.

The algebraic aspects of the multi-index based regularity structures
are worked out in \cite{LOT}, 
where in line with Hairer's postulates
the underlying Hopf-algebraic nature of the structure group was uncovered. 
In fact, the Hopf algebra arises
from a Lie algebra generated by
natural actions on the space of non-linearities $a$ and solutions $u$. 
As an introduction to the algebraic aspects in the theory of regularity structures 
we recommend \cite{che}.

\medskip

\noindent
Other approaches to singular SPDEs include the theory of paracontrolled
distributions by Gubinelli, Imkeller, and Perkowski, we recommend \cite{GP} 
for a first reading, and the renormalization group flow approach 
introduced by Kupiainen and generalized by Duch; we recommend \cite{Kup} 
and \cite{Duc} for an introduction. The para-controlled calculus provides
an alternative to the separation into model and modelled distribution,
replacing localization in physical space-time by localization on the Fourier side;
it is (typically) also indexed by trees.
The flow approach blends the stochastic and the deterministic step of regularity structures,
and has an index set closer to multi-indices. While these alternative approaches might
be more efficient in specific situations, they presumably lack the full flexibility 
of the two-pronged approach of regularity structures with its conceptual clarity.

\section{A singular quasi-linear SPDE}

\noindent
We are interested in nonlinear elliptic or parabolic equations with a random and thus 
typically rough right hand side $\xi$. 
Our approach is guided by moving beyond the well-studied semi-linear case. 
We consider a mildly quasi-linear case
where the coefficients of the leading-order derivatives depend on the solution $u$
itself. To fix ideas, we focus on the parabolic case in a single space dimension;
since we treat the parabolic equation in the whole space-time like an anisotropic elliptic
equation, we denote by $x_1$ the space-like and by $x_2$ the time-like variable. Hence
we propose to consider
\begin{align}\label{ao22}
(\partial_2-\partial_1^2)u=a(u)\partial_1^2u+\xi,
\end{align}
where we think of the values of $a(u)$ to be such 
that the equation is parabolic.
We are interested in laws / ensembles of $\xi$ where the
solutions $v$ to the linear equation
\begin{align}\label{ao25}
(\partial_2-\partial_1^2)v=\xi
\end{align}
are (almost surely) H\"older continuous, where it will turn out to be convenient
to express this in the ``annealed'' form\footnote{Think of Brownian motion which satisfies
$\mathbb{E}^\frac{1}{2}(B(s)-B(t))^2$ $=|s-t|^\frac{1}{2}$ while not being H\"older
continuous of exponent $\frac{1}{2}$ almost surely. Following the jargon annealed/quenched 
from statistical mechanics models (which itself is borrowed from metallurgy), 
we speak of annealed norms when the inner norm is an $L^p$-norm
w.~r.~t.~probability $\mathbb{E}$ and the outer norm is a space-time one.} of 
\begin{align}\label{cw37}
\sup_{x\not=y}\frac{1}{|y-x|^\alpha}\mathbb{E}^\frac{1}{2}|v(y)-v(x)|^2<\infty
\end{align}
for some exponent $\alpha\in(0,1)$.
In view of the anisotropic nature of $\partial_2-\partial_1^2$ and its invariance
under the rescaling $x_1=s\hat x_1$ and $x_2=s^2\hat x_2$, 
H\"older continuity in (\ref{cw37}) is measured w.~r.~t.~the Carnot-Carath\'eodory distance
\begin{align}\label{ao79}
``|y-x|\textnormal{''}:=\sqrt[4]{(y_1-x_1)^4+(y_2-x_2)^2}
\sim|y_1-x_1|+|y_2-x_2|^\frac{1}{2}.
\end{align}
By Schauder theory for $\partial_2-\partial_1^2$, 
on which we shall expand on in Subsection \ref{sec:Schauder},
this is the case for white noise $\xi$ with $\alpha=\frac{1}{2}$.
The rationale is that white noise has order of regularity $-\frac{D}{2}$,
where $D$ is the effective dimension, which in case of (\ref{ao25}) 
is $D=1+2=3$ since in view of (\ref{ao79})
the time-like variable $x_2$ counts twice, and that
$(\partial_2-\partial_1^2)^{-1}$ increases regularity by two, 
leading to $-\frac{D}{2}+2=\frac{1}{2}$.


In the range of $\alpha\in(0,1)$, the SPDE (\ref{ao22}) is what is called ``singular'': 
We cannot expect that the order of regularity of $u$ and thus
$a(u)$ is better than the one of $v$, which is $\alpha$, and hence the order
of regularity of $\partial_1^2u$ is no better than $\alpha-2$.
Since $\alpha+(\alpha-2)<0$ for $\alpha<1$,
the product $a(u)\partial_1^2u$ cannot be classically/deterministically defined\footnote{It is a classical result that the multiplication extends naturally from $C^\alpha \times C^\beta$ into $\mathcal{D}'$ if and only if $\alpha+\beta>0$, see \cite[Section 2.6]{BCD}.}. 
As discussed at the end of Section \ref{sec:Schauder},
a renormalization is needed\footnote{
The range $\alpha>1$, while still subtle for $\alpha<2$, 
does not require a renormalization, see \cite{LO}.}.


The same feature occurs for the (semi-linear) multiplicative heat equation
$(\partial_2-\partial_1^2)u=a(u)\xi$; in fact, our approach also applies to this 
semi-linear case, which already has been treated by (standard) regularity structures
in \cite{hp15}.
A singular product is already present in the case when the $x_1$-dependence
is suppressed, so that the above semi-linear equation turns into the SDE 
$\frac{du}{dx_2}=a(u)\xi$ with white noise $\xi$ in the time-like variable $x_2$. 
In this case, the analogue of $v$ from (\ref{ao25}) is Brownian motion, which
is characterized by $\mathbb{E}(v(y_2)-v(x_2))^2=|y_2-x_2|$ and
thus annealed H\"older exponent $\frac{1}{2}$ in $x_2$, which
in view of (\ref{ao79}) corresponds to the border-line setting $\alpha=1$.
Ito's integral and, more recently, Lyons' rough paths \cite{lyons} 
and Gubinelli's controlled rough paths \cite{gub}
have been devised to tackle the issue in this SDE setting.

\section{Annealed Schauder theory}\label{sec:Schauder}

\noindent
This section provides the main (linear) PDE ingredient for our result.
At the same time, it will allow us to discuss (\ref{ao25}).


In view of (\ref{ao25}), we are interested in the fundamental solution of the differential
operator $A:=\partial_2-\partial_1^2$. It turns out to be convenient to use the
more symmetric\footnote{It is symmetric under reflection 
not just in space but also in time} 
fundamental solution of the non-negative $A^*A$ $=(-\partial_2-\partial_1^2)
(\partial_2-\partial_1^2)$ $=\partial_1^4-\partial_2^2$. Moreover, it will
be more transparent to ``disintegrate'' the latter fundamental solution,
by which we mean writing it as $\int_0^\infty dt\psi_t(z)$, 
where $\{\psi_t\}_{t>0}$ are the kernels of the semi-group $\exp(-tA^*A)$.
Clearly, the Fourier transform is given by 
\begin{align}
{\mathcal F}\psi_t(q)=\exp(-t(q_1^4+q_2^2))\overset{\eqref{ao79}}{=}\exp(-t|q|^4).
\label{ft}
\end{align}
In particular, $\psi_{t}$ is a Schwartz function. 
For a Schwartz distribution $f$ like realizations of white noise,
we thus define $f_t(y)$ as the pairing of $f$ with $\psi_t(y-\cdot)$;
$f_t$ is a smooth function.
On the level of these kernels, the semi-group property translates into
\begin{align}\label{ao36}
\psi_s*\psi_{t}=\psi_{s+t}\quad\mbox{and}\quad\int\psi_t=1.
\end{align}
By construction, $\{\psi_t\}_t$ satisfies the PDE
\begin{align}\label{ao80}
\partial_t\psi_t+(\partial_1^4-\partial_2^2)\psi_t=0.
\end{align}
By scale invariance of \eqref{ao80} under $x_1=s \hat x_1$,
$x_2=s^2\hat x_2$, and $t=s^4\hat t$, we have
\begin{align}\label{ao37}
\psi_t(x_1,x_2)=\frac{1}{(\sqrt[4]{t})^{D=3}}\,\psi_1(\frac{x_1}{\sqrt[4]{t}},
\frac{x_2}{(\sqrt[4]{t})^2}).
\end{align}

\medskip

\begin{lemma}\label{lem:int}
Let $0<\alpha\le\eta<\infty$ with $\eta\not\in\mathbb{Z}$, 
$p<\infty$, and $x\in\mathbb{R}^2$ be given.
For a random Schwartz distribution $f$ with
\begin{align}\label{ao76}
\mathbb{E}^\frac{1}{p}|f_t(y)|^p\le(\sqrt[4]{t})^{\alpha-2}(\sqrt[4]{t}+|y-x|)^{\eta-\alpha}
\quad\mbox{for all}\;t>0,y\in\mathbb{R}^2,
\end{align}
there exists a unique random function $u$ of the class 
\begin{align}\label{ao55}
\sup_{y\in\mathbb{R}^2}\frac{1}{|y-x|^\eta}\mathbb{E}^\frac{1}{p}|u(y)|^p<\infty 
\end{align}
satisfying (distributionally in $\mathbb{R}^2$)
\begin{align}\label{ao56}
(\partial_2-\partial_1^2)u=f \quad ({\rm mod} \mbox{ polynomial of degree } \le\eta-2 ).
\end{align}
This unique solution $u$ actually satisfies $(\partial_2-\partial_1^2)u=f$.
Moreover, the l.~h.~s.~of \eqref{ao55} is bounded by a constant only depending on
$\alpha$ and $\eta$.
\end{lemma}


Now white noise $\xi$ is an example of such a random Schwartz distribution:
Since $\xi_t(y)$ is a centered Gaussian, we have $\mathbb{E}^\frac{1}{p}|\xi_t(y)|^p$
$\lesssim_p\mathbb{E}^\frac{1}{2}(\xi_t(y))^2$. By using the characterizing property
of white noise in terms of its pairing with a test function $\zeta$
\begin{align}
\mathbb{E}(\xi,\zeta)^2=\int\zeta^2,
\label{white}
\end{align}
we have $\mathbb{E}^\frac{1}{2}(\xi_t(y))^2$
$=\big(\int\psi_t^2(y-\cdot)\big)^\frac{1}{2}$, which by scaling (\ref{ao37})
is equal to $(\sqrt[4]{t})^{-\frac{D}{2}}(\int\psi_1^2)^\frac{1}{2}$
$\sim(\sqrt[4]{t})^{-\frac{D}{2}}$.
This specifies the sense in which white noise $\xi$ 
has order of regularity $-\frac{D}{2}$.


Fixing a ``base point'' $x$, Lemma \ref{lem:int} thus constructs the solution
of (\ref{ao25}) distinguished by $v(x)=0$. Note that the output (\ref{ao55}) takes the form of
$\mathbb{E}^\frac{1}{p}|v(y)-v(x)|^p\lesssim_p|y-x|^\frac{1}{2}$, which extends
(\ref{cw37}) from $p=2$ to general $p$.
Hence Lemma \ref{lem:int} provides an annealed version of a Schauder estimate,
alongside a Liouville-type uniqueness result.


\begin{proof}[Proof of Lemma \ref{lem:int}]
By construction, $\int_0^\infty dt(-\partial_2-\partial_1^2)\psi_t$ is the
fundamental solution of $\partial_2-\partial_1^2$, so that we take the
convolution of it with $f$. However, in order to obtain a convergent expression 
for $t\uparrow\infty$, we need to pass to a Taylor remainder:
\begin{align}\label{ao74}
u=\int_0^\infty dt({\rm id}-{\rm T}_x^\eta)(-\partial_2-\partial_1^2)f_t,
\end{align}
where ${\rm T}_x^\eta$ is the operation of taking the
Taylor polynomial of order $\le\eta$; as we shall argue the additional
Taylor polynomial does not affect the PDE. 

We claim that (\ref{ao74}) is well-defined and estimated as
\begin{align*}
\mathbb{E}^\frac{1}{p}|u(y)|^p\lesssim|y-x|^\eta.
\end{align*}
To this purpose, we first note that
\begin{align}\label{ao77}
\mathbb{E}^\frac{1}{p}|\partial^{\bf n}f_t(y)|^p\lesssim(\sqrt[4]{t})^{\alpha-2-|{\bf n}|}
(\sqrt[4]{t}+|y-x|)^{\eta-\alpha},
\end{align}
where
\begin{align}\label{ao26}
\partial^{\bf n}f:=\partial_1^{n_1}\partial_2^{n_2}f
\quad\mbox{and}\quad|{\bf n}|=n_1+2n_2.
\end{align}
Indeed, by the semi-group property (\ref{ao36}) we may write
$\partial^{\bf n}f_t(y)$ $=\int dz$ $\partial^{\bf n}\psi_\frac{t}{2}(y-z)$ $f_\frac{t}{2}(z)$,
so that
$\mathbb{E}^\frac{1}{p}|\partial^{\bf n}f_t(y)|^p$
$\le\int dz|\partial^{\bf n}\psi_\frac{t}{2}(y-z)|\mathbb{E}^\frac{1}{p}|f_\frac{t}{2}(z)|^p$.
Hence by (\ref{ao76}), (\ref{ao77}) follows from the kernel bound
$\int dz$ $|\partial^{\bf n}\psi_\frac{t}{2}(y-z)|$ $(\sqrt[4]{t}+|y-x|)^{\eta-\alpha}$
$\lesssim(\sqrt[4]{t})^{-|{\bf n}|}(\sqrt[4]{t}+|y-x|)^{\eta-\alpha}$, which itself
is a consequence of the scaling (\ref{ao37}) and the fact that $\psi_\frac{1}{2}$ 
is a Schwartz function.


Equipped with (\ref{ao77}), we now derive two estimates for the integrand of (\ref{ao74}),
namely for $\sqrt[4]{t}\ge|y-x|$ (``far field'') and for $\sqrt[4]{t}\le|y-x|$ (``near field'').
We write the Taylor remainder $({\rm id}-{\rm T}_x^\eta)(\partial_2+\partial_1^2)f_t(y)$
as a linear combination of\footnote{where $x^{\bf n}$ $:=x_1^{n_1}x_2^{n_2}$} 
$(y-x)^{\bf n}\partial^{\bf n}(\partial_2+\partial_1^2)f_t(z)$
with $|{\bf n}|>\eta$ and
at some point $z$ intermediate to $y$ and $x$. By (\ref{ao77}) such a term is estimated
by $|y-x|^{|{\bf n}|}(\sqrt[4]{t})^{\alpha-4-|{\bf n}|}(\sqrt[4]{t}+|y-x|)^{\eta-\alpha}$,
which in the far field is 
$\sim |y-x|^{|{\bf n}|}(\sqrt[4]{t})^{\eta-4-|{\bf n}|}$. Since the exponent on $t$ is $<-1$,
we obtain as desired
\begin{align*}
\mathbb{E}^\frac{1}{p}|\int_{|y-x|^4}^\infty dt
({\rm id}-{\rm T}_x^\eta)(\partial_2+\partial_1^2)f_t(y)|^p\lesssim |y-x|^{\eta}.
\end{align*}


For the near-field term, i.~e.~for $\sqrt[4]{t}\le|y-x|$, we proceed as 
follows:
\begin{align*}
\lefteqn{\mathbb{E}^\frac{1}{p}|({\rm id}-{\rm T}_x^\eta)(\partial_2+\partial_1^2)f_t(y)|^p}
\nonumber\\
&\le\mathbb{E}^\frac{1}{p}|(\partial_2+\partial_1^2)f_t(y)|^p
+\sum_{|{\bf n}|\le\eta}|y-x|^{|\bf n|}
\mathbb{E}^\frac{1}{p}|\partial^{\bf n}(\partial_2+\partial_1^2)f_t(x)|^p\nonumber\\
&\overset{\eqref{ao77}}{\lesssim}
(\sqrt[4]{t})^{\alpha-4}|y-x|^{\eta-\alpha}
+\sum_{|{\bf n}|\le\eta}|y-x|^{|{\bf n}|}(\sqrt[4]{t})^{\eta-4-|{\bf n}|}.
\end{align*}
Since $\eta$ is not an integer, the sum restricts to $|{\bf n}|<\eta$,
so that all exponents on $t$ are $>-1$. Hence we obtain as desired
\begin{align*}
\mathbb{E}^\frac{1}{p}|\int_0^{|y-x|^4}dt
({\rm id}-{\rm T}_x^\eta)(\partial_2+\partial_1^2)f_t(y)|^p
\lesssim|y-x|^\eta.
\end{align*}
It can be easily checked that \eqref{ao74} 
is indeed a solution of \eqref{ao56}, even without a polynomial. 
For a detailed proof we refer to \cite[Proposition~4.3]{LOTT}. 


We turn to the uniqueness of $u$ in the class \eqref{ao55} satisfying \eqref{ao56}. 
Given two such solutions $u_1,u_2$, we observe that $\bar{u}:=u_1-u_2$ 
satisfies \eqref{ao55} and \eqref{ao56} with $f=0$. 
In particular $\partial^\n(\partial_2-\partial_1^2)\bar{u}=0$ for $|\n|>\eta-2$, 
and thus from \eqref{ao80} we obtain $\partial_t\partial^\n\bar{u}_t=0$ 
provided $|\n|>\eta-4$. Thus, $\partial^\n \bar{u}_t$ is independent of $t>0$.
Moreover, \eqref{ao55} implies that $\mathbb{E}|\partial^\n\bar{u}_t|\to 0$ as 
$t\to\infty$ for $|\n|>\eta$. 
Hence we learn from $t\to0$ that $\partial^\n\bar{u}=0$ for $|\n|>\eta$, 
i.e. $\bar{u}$ is a polynomial of degree $\leq\eta$. 
Since $\eta\not\in\mathbb{Z}$ this strengthens to $\bar{u}$ is a polynomial of degree $<\eta$, 
and by \eqref{ao55} it vanishes at $x$ to order $\eta$ which yields the desired $\bar{u}=0$.
\end{proof}


We return to the discussion of the singular product $a(u)\partial_1^2u$, 
in its simplest form of
\begin{align*}
v\partial_1^2v=\partial_1^2\frac{1}{2}v^2-(\partial_1v)^2.
\end{align*}
While in view of Lemma \ref{lem:int} the first r.~h.~s.~term 
is well-defined as a random Schwartz distribution,
we now argue that the second term diverges. 
Indeed, applying $\partial_1$ to 
the representation formula (\ref{ao74}), so that the constant Taylor term drops out, we have 
\begin{align}\label{ao30}
\partial_1 v=\int_0^\infty dt \partial_1(-\partial_2-\partial_1^2)\xi_t.
\end{align}
Hence for space-time white noise 
\begin{align*}
&\mathbb{E}(\partial_1 v(x))^2 \\ 
&\overset{\eqref{ao30}}{=}\int_0^\infty dt \int_0^\infty ds \, \mathbb{E}\Big(\partial_1(-\partial_2-\partial_1^2)\xi_t(x) \partial_1(-\partial_2-\partial_1^2)\xi_s(x)  \Big) \\
&\overset{\eqref{white}}{=}\int_0^\infty dt \int_0^\infty ds \int_{\mathbb{R}^2}dy \, 
\partial_1(-\partial_2-\partial_1^2)\psi_t(x-y) 
\partial_1(-\partial_2-\partial_1^2)\psi_s(x-y) \\
&\overset{\eqref{ao36}}{=}\int_0^\infty dt \int_0^\infty ds \, 
\partial_1^2(\partial_2^2-\partial_1^4) \psi_{s+t}(0) \\
&\overset{\eqref{ao37}}{\sim} \int_0^\infty dt \int_0^\infty ds\,
\sqrt[4]{t+s}^{-D-6}. 
\end{align*}
Note that since $\frac{1}{4}(-D-6)<-2$ for $D=3$, the double integral diverges.
This divergence arises from $t\downarrow 0$ and $s\downarrow 0$, that is, from small space-time scales, and thus is called an ultra-violet (UV) divergence. A quick fix is to introduce an UV cut-off, which for instance can be implemented
by mollifying $\xi$. Using the semi-group convolution $\xi_\tau$ specifies the
UV cut-off scale to be of the order of $\sqrt[4]{\tau}$. It is easy to check
that in this case
\begin{align*}
\mathbb{E}(\partial_1v(x))^2\sim \int_\tau^\infty dt \int_\tau^\infty ds\,
\sqrt[4]{t+s}^{-D-6}\sim(\sqrt[4]{\tau})^{-1}.
\end{align*}


The goal is to modify the equation (\ref{ao22}) by ``counter terms'' such that
\begin{itemize}
\item the solution manifold stays under control 
as the ultra-violet cut-off $\tau\downarrow 0$, 
\item invariances of the solution manifold are preserved i.e. the solution manifold keeps as many symmetries as possible. 
\end{itemize}
In view of the above
discussion, we expect the coefficients of the counter terms to diverge as the 
cut-off tends to zero.
 

\section{Symmetry-motivated postulates on the counter terms}\label{sec:post}

\noindent
In view of $\alpha\in(0,1)$, $u$ is a function while we think of all derivatives $\partial^{\bf n}u$
as being only Schwartz distributions. Hence it is natural to start from the very general
Ansatz that the counter term is
a polynomial in $\{\partial^{\bf n}u\}_{{\bf n}\not={\bf 0}}$ 
with coefficients that are general (local) functions in $u$:
\begin{align}\label{ao23}
(\partial_2-\partial_1^2)u+\sum_{\beta}h_{\beta}(u)\prod_{{\bf n}\neq {\bf 0}}
(\frac{1}{{\bf n}!}\partial^{\bf n}u)^{\beta({\bf n})}
=a(u)\partial_1^2u+\xi,
\end{align}
where $\beta$ runs over all multi-indices\footnote{which associate to every
index ${\bf n}$ a $\beta({\bf n})\in\mathbb{N}_0$ such that $\beta({\bf n})$ vanishes
for all but finitely many ${\bf n}$'s} in ${\bf n}\not={\bf 0}$ and 
${\bf n}!:=(n_1!)(n_2!)$.
For simplicity of this heuristic discussion, 
we drop the regularization on $\xi$ and don't index the counter term with $\tau$.


Only counter terms that have an order strictly below the order of the leading 
$\partial_2-\partial_1^2$ are desirable,
so that one postulates that the sum in (\ref{ao23})
restricts to those multi-indices for which
\begin{align}\label{cw14}
|\beta|_p:=\sum_{{\bf n}\not={\bf 0}}|{\bf n}|\beta({\bf n})<2
\quad\mbox{where}\quad|{\bf n}|:=n_1+2n_2.
\end{align}
This leaves only $\beta=0$ and $\beta=e_{(1,0)}$, where the latter means $\beta({\bf n})=\delta_{\bf n}^{(1,0)}$, so that (\ref{ao23}) collapses to
\begin{align}\label{ao24}
(\partial_2-\partial_1^2)u+h(u)+h'(u)\partial_1u=a(u)\partial_1^2u+\xi.
\end{align}
One also postulates that $h$ and $h'$ depend on the noise $\xi$ only through its
law / distribution / ensemble, hence are deterministic. 
Since we assume that the law is invariant under space-time
translation, i.~e.~is stationary,
it was natural to postulate that $h$ and $h'$ do not explicitly depend on $x$,
hence are homogeneous.

\medskip

\noindent 
{\sc Reflection symmetry}. Let us now assume that the law of $\xi$ is invariant under
\begin{equation}\label{ao30bis}
\begin{split}
\textnormal{space-time translation } y &\mapsto y+x, \\
\textnormal{space reflection } y &\mapsto (-y_1,y_2).
\end{split}
\end{equation}
We now argue that under this assumption, it is natural to postulate that the term 
$h'(u)\partial_1u$ in (\ref{ao24}) is not present, so that we are left with
\begin{align}\label{ao27}
(\partial_2-\partial_1^2)u+h(u)=a(u)\partial_1^2u+\xi.
\end{align}
To this purpose, let 
$x\in\mathbb{R}^2$ be arbitrary yet fixed, and consider the reflection at the line $\{y_1=x_1\}$
given by $R y=(2x_1-y_1,y_2)$, which by pull back acts on functions as $\tilde u(y):= u(Ry)$. 
Since (\ref{ao22})
features no explicit $y$-dependence, and only involves even powers of $\partial_1$,
which like $\partial_2$ commute with $R$, we have
\begin{align}\label{ao29}
(u,\xi)\;\mbox{satisfies (\ref{ao22})}\quad\Longrightarrow\quad
(u(R\cdot),\xi(R\cdot))\;\mbox{satisfies (\ref{ao22})}.
\end{align}
Since we postulated that $h$ and $h'$ depend on $\xi$ only
via its law, and since in view of the assumption (\ref{ao30bis}),
$\tilde\xi=\xi(R\cdot)$ has the same law as $\xi$, it is natural to postulate
that the symmetry (\ref{ao29}) extends from (\ref{ao22}) to (\ref{ao24}). Spelled out,
this means that (\ref{ao24}) implies
\begin{align*}
(\partial_2-\partial_1^2)\tilde u+h(\tilde u)+h'(\tilde u)\partial_1\tilde u
=a(\tilde u)\partial_1^2\tilde u+\tilde\xi.
\end{align*}
Evaluating both identities at $y=x$, and taking the difference, we get for any
solution of (\ref{ao24}) that $h'(u(x))\partial_1u(x)$ $=h'(u(x))(-\partial_1u(x))$,
and thus $h'(u(x))\partial_1u(x)=0$, as desired.

\medskip

\noindent
{\sc Covariance under $u$-shift}.
We now come to our most crucial postulate, which restricts how the nonlinearity
$h$ depends on the nonlinearity / constitutive law $a$. Hence we no longer
think of a single nonlinearity $a$, but consider all non-linearities at once, 
in the spirit of rough paths.
This point of view reveals another invariance of (\ref{ao22}), namely 
for any shift $v\in\mathbb{R}$ 
\begin{align}\label{ao32}
(u,a)\;\mbox{satisfies (\ref{ao22})}\quad\Longrightarrow\quad
(u-v,a(\cdot+v))\;\mbox{satisfies (\ref{ao22})}.
\end{align}
A priori, $h$ is a function of
the $u$-variable that has a functional dependence on $a$, as denoted by $h=h[a](u)$.
We postulate that the symmetry (\ref{ao32}) extends from (\ref{ao22}) to (\ref{ao27}).
This is the case provided we have the following shift-covariance property
\begin{align}\label{ao04}
h[a](u+v)=h[a(\cdot+v)](u)\quad\mbox{for all}\;u\in\mathbb{R}.
\end{align}
This property can also be paraphrased as: Whatever algorithm one uses to construct $h$
from $a$,
it should not depend on the choice of origin in what is just an affine space $\mathbb{R}\ni u$.
Property (\ref{ao04}) implies that the counter term is determined by a functional $c=c[a]$
on the space of nonlinearities $a$:
\begin{align}\label{ao09}
h[a](v)=c[a(\cdot+v)].
\end{align}
Renormalization now amounts to choosing $c$ such that the solution manifold
stays under control as the UV regularization of $\xi$ fades away.


\section{Algebrizing the counter term}\label{ss:3.1}

\noindent
In this section, we algebrize the relationship between $a$ and
the counter term $h$ given by a functional $c$ as in (\ref{ao09}). 
To this purpose, we introduce the following coordinates\footnote{
where here and in the sequel $k\ge 0$ stands short for $k\in\mathbb{N}_0$} on the space
of analytic functions $a$ of the variable $u$:
\begin{align}\label{ao11}
\mathsf{z}_k[a]:=\frac{1}{k!}\frac{d^ka}{du^k}(0)\quad\mbox{for}\;k\ge 0.
\end{align}
These are made such that by Taylor's theorem
\begin{align}\label{ao02}
a(u)=\sum_{k\ge 0}u^k\mathsf{z}_k[a]\quad\mbox{for}\;a\in\mathbb{R}[u],
\end{align}
where $\mathbb{R}[u]$ denotes the algebra of polynomials in the single variable
$u$ with coefficients in $\mathbb{R}$. 


We momentarily specify to functionals $c$
on the space of analytic $a$'s that can be represented as polynomials in the (infinitely many)
variables $\{\mathsf{z}_k\}_{k\ge 0}$. 
This leads us to consider the algebra $\mathbb{R}[\mathsf{z}_k]$
of polynomials in the variables $\mathsf{z}_k$ with coefficients in $\mathbb{R}$. 
The monomials 
\begin{align}\label{ao14}
\mathsf{z}^\beta:=\prod_{k\ge 0}\mathsf{z}_k^{\beta(k)}
\end{align}
form a basis of this (infinite dimensional)
linear space, where $\beta$ runs over all multi-indices\footnote{which
means they associate a frequency $\beta(k)\in\mathbb{N}_0$ to every $k\ge 0$
such that all but finitely many $\beta(k)$'s vanish}. Hence as a linear space,
$\mathbb{R}[\mathsf{z}_k]$ can be seen as the direct sum over the index set given
by all multi-indices $\beta$, and we think of $c$ as being of the form
\begin{align}\label{ao16}
c[a]=\sum_{\beta}c_\beta\mathsf{z}^\beta[a]
\quad\mbox{for}\;c\in\mathbb{R}[\mathsf{z}_k].
\end{align}

\medskip

\noindent
{\sc Infinitesimal $u$-shift}. Given a shift $v\in\mathbb{R}$, for $\tilde u:=u-v$
and $\tilde a:=a(\cdot+v)$ we have $\tilde a(\tilde u)=a(u)$.
This leads us to study the mapping $a\mapsto a(\cdot+v)$ which 
provides an action/representation of the group $\mathbb{R}\ni v$
on the set $\mathbb{R}[u]\ni a$.
Note that for 
$c\in\mathbb{R}[\mathsf{z}_k]$ and $a\in\mathbb{R}[u]$,
the function $\mathbb{R}\ni v\mapsto c[a(\cdot+v)]=\sum_{\beta}c_\beta\prod_{k\ge 0}
(\frac{1}{k!}\frac{d^ka}{du}(v))^{\beta(k)}$ is polynomial. Thus 
\begin{align}\label{ao06}
(D^{({\bf 0})}c)[a]=\frac{d}{dv}_{|v=0}c[a(\cdot+v)]
\end{align}
is well-defined, linear in $c$ and even a derivation\footnote{the index $({\bf 0})$ is 
not necessary for these lecture notes, 
since we do not appeal to the other derivations $\{D^{({\bf n})}\}_{{\bf n}\not={\bf 0}}$ 
from \cite{LOT,LOTT}, we keep it here for consistency with these papers}, 
meaning that Leibniz' rule holds
\begin{align}\label{ao15}
(D^{({\bf 0})}cc')=(D^{({\bf 0})}c)c'+c(D^{({\bf 0})}c').
\end{align}
The latter implies that $D^{({\bf 0})}$ is determined by its value on the
coordinates $\mathsf{z}_k$, which by definitions (\ref{ao11}) and (\ref{ao06})
is given by $D^{({\bf 0})}\mathsf{z}_k$ $=(k+1)\mathsf{z}_{k+1}$. Hence $D^{({\bf 0})}$ has
to agree with the following derivation on the algebra $\mathbb{R}[\mathsf{z}_k]$
\begin{align}\label{ao13}
D^{({\bf 0})}=\sum_{k\ge 0}(k+1)\mathsf{z}_{k+1}\partial_{\mathsf{z}_k},
\end{align}
which is well defined since the sum is effectively finite when applied to a monomial.

\medskip

\noindent
{\sc Representation of counter term}. 
Iterating (\ref{ao06}) we obtain by induction in $l\ge 0$
for $c\in\mathbb{R}[\mathsf{z}_k]$ and $a\in\mathbb{R}[u]$
\begin{align*}
\frac{d^l}{dv^l}_{|v=0}c[a(\cdot+v)]=((D^{({\bf 0})})^lc)[a]
\end{align*}
and thus by Taylor's theorem (recall that $v\mapsto c[a(\cdot+v)]$ is polynomial)
\begin{align}\label{ao07}
c[a(\cdot+v)]=\big(\sum_{l\ge 0}\frac{1}{l!}v^l(D^{({\bf 0})})^lc\big)[a].
\end{align}
We combine (\ref{ao07}) with (\ref{ao09}) to obtain the representation
\begin{align}\label{cw11}
h[a](v)=\big(\sum_{l\ge 0}\frac{1}{l!}v^l(D^{({\bf 0})})^lc\big)[a].
\end{align}
Hence our goal is to determine the coefficients $\{c_\beta\}_{\beta}$
in (\ref{ao16}), which typically will blow up as $\tau\downarrow 0$.

%

\section{Algebrizing the solution manifold: The centered model}\label{ss:3.2}

\noindent
The purpose of this section is to motivate the notion of a centered model;
the motivation will be in parts informal.

\medskip

\noindent
{\sc Parameterization of the solution manifold}.
If $a\equiv 0$ it follows from (\ref{ao04}) that $h$ is a (deterministic) constant.
We learned from the discussion after Lemma \ref{lem:int} that -- given a base point $x$ --
there is a distinguished solution $v$ (with $v(x)=0$). Hence we may {\it canonically}
parameterize a general solution $u$ of (\ref{ao27}) for $a\equiv 0$ via $u=v+p$, by space-time functions $p$
with $(\partial_2-\partial_1^2)p=0$. Such $p$ are necessarily analytic. 
Having realized this, it is convenient\footnote{otherwise, the coordinates $\mathsf{z}_{(2,0)}$
and $\mathsf{z}_{(0,1)}$ defined in (\ref{ao48}) would be redundant on $p$-space} 
to free oneself from the constraint 
$(\partial_2-\partial_1^2)p=0$, which can be done at the expense of relaxing
(\ref{ao27}) to
\begin{align}\label{ao43}
(\partial_2-\partial_1^2)v=\xi\quad ({\rm mod}\mbox{ analytic space-time functions}).
\end{align}
Since we think of $\xi$ as being rough while analytic functions are infinitely smooth, 
this relaxation is still constraining $v$.


The implicit function theorem suggests that this parameterization (locally) persists 
in the presence of a sufficiently small analytic nonlinearity $a$: The nonlinear
manifold of all space-time functions $u$ that satisfy
\begin{align}\label{ao45}
(\partial_2-\partial_1^2)u+h(u)=a(u)\partial_1^2u+\xi \ 
({\rm mod}\mbox{ analytic space-time functions})
\end{align}
is still parameterized by space-time analytic functions $p$. We now return to the
point of view of Section \ref{sec:post} 
of considering all nonlinearities $a$ at once, meaning that
we consider the (still nonlinear) space of all space-time functions that satisfy
(\ref{ao45}) for {\it some} analytic nonlinearity $a$. We want to
capitalize on the symmetry (\ref{ao32}), which extends from (\ref{ao22}) to
(\ref{ao27}) and to (\ref{ao45}). 
We do so by considering the above space of $u$'s {\it modulo constants},
which we implement by focusing on increments $u-u(x)$. 
Summing up, it is reasonable to expect that the space of all space-time functions
$u$, modulo space-time constants, that satisfy (\ref{ao45}) for some analytic nonlinearity $a$
(but at fixed $\xi$), 
is parameterized by pairs $(a,p)$ with $p(x)=0$.

\medskip

\noindent
{\sc Formal series representation}. 
In line with the term-by-term approach from physics, we write the increment
$u(y)-u(x)$ as a (typically divergent) power series
\begin{align}\label{ao83}
\lefteqn{u(y)-u(x)}\nonumber\\
&=\sum_{\beta}\Pi_{x\beta}(y)
\prod_{k\ge 0}\big(\frac{1}{k!}\frac{d^ka}{du^k}(u(x))\big)^{\beta(k)}
\prod_{{\bf n}\not={\bf 0}}\big(\frac{1}{{\bf n}!}\partial^{\bf n}p(x)\big)^{\beta({\bf n})},
\end{align}
where $\beta$ runs over all multi-indices in $k\ge 0$ and ${\bf n}\not={\bf 0}$.
Introducing coordinates on the space
of analytic space-time functions $p$ with $p(0)=0$ via\footnote{where here and in the
sequel ${\bf n}\not={\bf 0}$ stands short for ${\bf n}\in\mathbb{N}_0^2-\{(0,0)\}$}
\begin{align}\label{ao48}
\mathsf{z}_{\bf n}[p]=\frac{1}{{\bf n}!}\partial^{\bf n}p(0)\quad
\mbox{for}\;{\bf n}\not={\bf 0},
\end{align}
(\ref{ao83}) can be more compactly written as
\begin{align}\label{ao01}
u(y)=
u(x)+\sum_{\beta}\Pi_{x\beta}(y)\mathsf{z}^\beta[a(\cdot+u(x)),p(\cdot+x)-p(x)].
\end{align}
This is reminiscent of Butcher series in the analysis of ODE discretizations.
 

Recall from above that for $a\equiv 0$ we have the explicit parameterization
\begin{align}\label{cw13}
u[a=0,p]-u[a=0,p](x)=v+p
\end{align}
with the distinguished solution $v$ of the linear equation.
Hence from setting $a\equiv 0$ and $p\equiv 0$ in (\ref{ao83}), 
we learn 
\begin{equation*}
v 
\overset{\eqref{cw13}}{=} u[a=0,p=0]-u[a=0,p=0](x)
\overset{\eqref{ao83}}{=} \sum_\beta \Pi_{x\beta} \prod_{k\geq0} 0^{\beta(k)} \prod_{{\bf n}\neq{\bf 0}} 0^{\beta({\bf n})} 
\end{equation*}
and thus\footnote{we use the convention that $0^0=1$} $v=\Pi_{x0}$. 
Similarly, from keeping $a\equiv 0$ but letting $p$ vary, we obtain
\begin{equation*}
v+p
\overset{\eqref{cw13}}{=} u[a\hspace{-.5ex}=\hspace{-.5ex}0,p]-u[a\hspace{-.5ex}=\hspace{-.5ex}0,p](x)
\overset{\eqref{ao83}}{=}
 \sum_{\beta} \Pi_{x\beta} \prod_{k\geq0} 0^{\beta(k)} \hspace{-.5ex}
\prod_{{\bf n}\neq{\bf 0}} \big(\tfrac{1}{{\bf n}!} \partial^{\bf n} p(x)\big)^{\beta({\bf n})} ,
\end{equation*}
hence we deduce
that for all multi-indices $\beta\not=0$ which satisfy $\beta(k)=0$ for all $k\ge 0$ we 
must have\footnote{where we recall that $\beta=e_{\bf n}$ denotes the multi-index
with $\beta({\bf m})=\delta_{\bf m}^{\bf n}$ next to $\beta(k)=0$ for all $k$}
\begin{align}\label{ao59}
\Pi_{x\beta}(y)=\left\{\begin{array}{cc}
(y-x)^{\bf n}&\mbox{provided}\;\beta=e_{\bf n}\\
0&\mbox{else}\end{array}\right\}.
\end{align}

\medskip

\noindent
{\sc Hierarchy of linear equations}. 
The collection $\{\Pi_{x\beta}(y)\}_{\beta}$ of coefficients from (\ref{ao01})
is an element of the direct {\it product} with the same index set 
as the direct {\it sum} $\mathbb{R}[\mathsf{z}_k,\mathsf{z}_{\bf n}]$. 
Hence the direct product inherits the multiplication of the polynomial algebra
\begin{align}\label{ao52}
(\pi\pi')_{\bar\beta}=\sum_{\beta+\beta'=\bar\beta}\pi_\beta\pi'_{\beta'},
\end{align}
and is denoted as the (well-defined) algebra $\mathbb{R}[[\mathsf{z}_k,\mathsf{z}_{\bf n}]]$
of formal power series; we denote by $\mathsf{1}$ its unit element. 
We claim that in terms of (\ref{ao01}), 
(\ref{ao45}) assumes the form of
\begin{align}\label{cw09}
(\partial_2-\partial_1^2)\Pi_x=\Pi_{x}^{-}\quad ({\rm mod}\mbox{ analytic space-time functions})
\end{align}
where
\begin{align}\label{ao49}
\Pi_{x}^{-}:=\sum_{k\ge 0}\mathsf{z}_k\Pi_x^k\partial_1^2\Pi_x
-\sum_{l\ge 0}\frac{1}{l!}\Pi_x^l(D^{({\bf 0})})^lc+\xi_\tau\mathsf{1},
\end{align}
as an identity in formal power series in $\mathsf{z}_k,\mathsf{z}_{\bf n}$
with coefficients that are continuous space-time functions. 
We shall argue below that (\ref{ao49}) is effectively, i.~e.~componentwise,
well-defined despite the two infinite sums, and despite extending from 
$c\in\mathbb{R}[\mathsf{z}_k]$ to $c\in\mathbb{R}[[\mathsf{z}_k]]$.
Moreover, as will become clear by \eqref{ks05}, 
the $\beta$-component of \eqref{ao49} contains on the r.~h.~s.~only 
terms $\Pi_{x\beta'}$ for ``preceding'' multi-indices $\beta'$  
-- hence \eqref{cw09} describes a \emph{hierarchy} of equations.


Here comes the informal argument for \eqref{ao49}, 
relating $\{\partial_2,\partial_1^2\}u$, 
$a(u)$ and
$h(u)$ to $\{\partial_2,\partial_1^2\}\Pi_x[\tilde a,\tilde p]$,
$(\sum_{k}\mathsf{z}_k\Pi_x^k)[\tilde a,\tilde p]$ 
and $(\sum_{l}\frac{1}{l!}\Pi_x^l(D^{({\bf 0})})^lc)
[\tilde a,\tilde p]$, respectively. Here 
we have set for abbreviation $\tilde a$ $=a(\cdot+u(x))$ and $\tilde p$ $=p(\cdot+x)-p(x)$.
It is based on (\ref{ao01}), which can be compactly written as
$u(y)=u(x)+\Pi_x[\tilde a,\tilde p](y)$.
Hence the statement on $\{\partial_2,\partial_1^2\}u$ follows immediately.
Together with $a(u(y))$ $=\tilde a(u(y)-u(x))$, this also implies by (\ref{ao02}) the desired
\begin{align*}
a(u(y))=\big(\sum_{k\ge 0}\mathsf{z}_k\Pi_x^k(y)\big)[\tilde a,\tilde p].
\end{align*}
Likewise by (\ref{ao04}), we have $h[a](u(y))$ $=h[\tilde a](u(y)-u(x))$, so that
by (\ref{cw11}), we obtain the desired
\begin{align*}
h[a](u(y))=\big(\sum_{l\ge 0}\frac{1}{l!}\Pi_x^l(y)(D^{({\bf 0})})^lc\big)[\tilde a,\tilde p].
\end{align*}

\medskip

\noindent
{\sc Finiteness properties}.
The next lemma collects crucial algebraic properties.

\begin{lemma}\label{lem:alg}
The derivation $D^{({\bf 0})}$ extends from $\mathbb{R}[\mathsf{z}_k]$ to
$\mathbb{R}[[\mathsf{z}_k]]$. 
Moreover, for $\pi,\pi'\in\mathbb{R}[[\mathsf{z}_k,\mathsf{z}_{\bf n}]]$,
$c\in\mathbb{R}[[\mathsf{z}_k]]$, and $\xi\in\mathbb{R}$,
\begin{align}\label{cw07}
\pi^{-}:=\sum_{k\ge 0}\mathsf{z}_k\pi^k\pi'-\sum_{l\ge 0}\frac{1}{l!}\pi^l(D^{({\bf 0})})^lc
+\xi\mathsf{1}
\;\in\;\mathbb{R}[[\mathsf{z}_k,\mathsf{z}_{\bf n}]]
\end{align}
is well-defined, in the sense that the two sums are componentwise finite.
Finally, for 
\begin{align}\label{cw15}
[\beta]:=\sum_{k\ge 0}k\beta(k)-\sum_{{\bf n}\not={\bf 0}}\beta({\bf n})
\end{align}
we have the implication
\begin{align}\label{cw08}
&\pi_\beta=\pi'_\beta=0\quad\mbox{unless}\quad[\beta]\ge 0\;\mbox{or}\;\beta=e_{\bf n}
\;\mbox{for some}\;{\bf n}\not={\bf 0}\nonumber\\
&\Longrightarrow\nonumber\\
&\pi^-_\beta=0\quad\mbox{unless}\quad\left\{\begin{array}{l}
[\beta]\ge 0\;\mbox{or}\\
\beta=e_k+e_{{\bf n}_1}+\cdots+e_{{\bf n}_{k+1}}\\
\mbox{for some}\;k\ge 1\;
\mbox{and}\;{\bf n}_1,\dots,{\bf n}_{k+1}\not={\bf 0}.
\end{array}\right\}.
\end{align}
\end{lemma}

We note that for $\beta$ as in the second alternative on the r.~h.~s.~of
(\ref{cw08}), it follows from (\ref{ao59}) that $\Pi_{x\beta}^{-}$ is
a polynomial. Hence in view of the modulo in (\ref{cw09}), we learn
from (\ref{cw08}) that we may assume
\begin{align}\label{cw03}
\Pi_{x\beta}\equiv 0\quad
\mbox{unless}\quad[\beta]\ge 0\;\;\mbox{or}\;\;\beta=e_{\bf n}\;\mbox{for some}\;{\bf n}\not={\bf 0}.
\end{align}


\begin{proof}[Proof of Lemma \ref{lem:alg}]
We first address the extension of $D^{({\bf 0})}$
and note that from (\ref{ao13}) we may read off
the matrix representation of $D^{({\bf 0})}$
$\in{\rm End}(\mathbb{R}[\mathsf{z}_k])$ w.~r.~t.~(\ref{ao14}) given by
\begin{align}\label{ao20}
\lefteqn{(D^{({\bf 0})})_\beta^\gamma=(D^{({\bf 0})}\mathsf{z}^\gamma)_\beta
\overset{\eqref{ao13}}{=}
\sum_{k\ge 0}(k+1)\big(\mathsf{z}_{k+1}\partial_{\mathsf{z}_k}\mathsf{z}^\gamma\big)_\beta
}\nonumber\\
&\overset{\eqref{ao14}}{=}\sum_{k\ge 0}(k+1)\gamma(k)\left\{\begin{array}{cl}
1&\mbox{provided}\;\gamma+e_{k+1}=\beta+e_k\\
0&\mbox{otherwise}\end{array}\right\}.
\end{align}
From this we read off that $\{\gamma|(D^{({\bf 0})})_\beta^\gamma\not=0\}$
is finite for every $\beta$, which implies that $D^{({\bf 0})}$ naturally extends from
$\mathbb{R}[\mathsf{z}_k]$ to $\mathbb{R}[[\mathsf{z}_k]]$. 
With help of (\ref{ao52}) the derivation property (\ref{ao15}) can be expressed coordinate-wise, 
and thus extends to $\mathbb{R}[[\mathsf{z}_k]]$.


We now turn to (\ref{cw07}), which component-wise reads
\begin{align}\label{ao51}
\pi_{\beta}^{-}&=\sum_{k\ge 0}\sum_{e_k+\beta_1+\cdots+\beta_{k+1}=\beta}
\pi_{\beta_1}\cdots\pi_{\beta_k}\pi_{\beta_{k+1}}'\nonumber\\
&-\sum_{l\ge 0}\frac{1}{l!}\sum_{\beta_1+\cdots+\beta_{l+1}=\beta}
\pi_{\beta_1}\cdots\pi_{\beta_l}((D^{({\bf 0})})^lc)_{\beta_{l+1}}
+\xi\delta_{\beta}^0,
\end{align}
and claim that the two sums are effectively finite. For the first term of the r.~h.~s.~this
is obvious since thanks to the presence of\footnote{$\gamma=e_k$ denotes the
multi-index with $\gamma(l)=\delta_l^k$ next to $\gamma({\bf n})=0$} $e_k$ in
$e_k+\beta_1+\cdots+\beta_{k+1}=\beta$, for fixed $\beta$
there are only finitely many $k\ge 0$ for which this relation can be satisfied.


In preparation for the second r.~h.~s.~term of (\ref{ao51}) we now establish that
\begin{align}\label{ao19}
((D^{({\bf 0})})^l)_\beta^\gamma=0\quad\mbox{unless}\quad[\beta]_0=[\gamma]_0+l,
\end{align}
where we introduced the scaled length $[\gamma]_0:=\sum_{k\ge 0}k\gamma(k)\in\mathbb{N}_0$. 
The argument for (\ref{ao19}) proceeds
by induction in $l\ge 0$. It is tautological for the base case $l=0$. In order
to pass from $l$ to $l+1$ we write
$((D^{({\bf 0})})^{l+1})_\beta^\gamma$ $=\sum_{\beta'}((D^{({\bf 0})})^{l})_\beta^{\beta'}
(D^{({\bf 0})})_{\beta'}^\gamma$; by induction hypothesis, the first factor vanishes
unless $[\beta]_0=[\beta']_0+l$. We read off (\ref{ao20}) that the second factor
vanishes unless $[\beta']_0=[\gamma]_0+1$, so that the product vanishes unless
$[\beta]_0=[\gamma]_0+(l+1)$, as desired.


Equipped with (\ref{ao19}) we now turn to the second r.~h.~s.~term of (\ref{ao51})
and note that
$((D^{({\bf 0})})^lc)_{\beta_{k+1}}$ vanishes unless $l\le[\beta_{k+1}]_0\le[\beta]_0$,
which shows that also here, only finitely many $l\ge 0$ contribute for fixed $\beta$.


We turn to the proof of \eqref{cw08}. 
We use \eqref{ao51} and give the proof for every summand separately. 
For the first term on the r.~h.~s.~of \eqref{ao51} we obtain by additivity of $[\cdot]$ that
$[\beta]=k+[\beta_1]+\cdots+[\beta_{k+1}]$. Note that $\pi_{\beta_i}$ is only 
non vanishing if $[\beta_i]\geq-1$. If at least one of the $\beta_1,\dots,\beta_{k+1}$ satisfies 
$[\beta_i]\geq0$, we obtain therefore $[\beta]\geq k-k = 0$.
For the second r.~h.~s.~term in \eqref{ao51} we appeal to \eqref{ao19}:
Since $D^{(\0)}$ doesn't affect the $\z_\n$ components, \eqref{ao19} extends 
from $[\cdot]_0$ to $[\cdot]$. Together with $c\in\mathbb{R}[[\z_k]]$ 
this yields $[\beta_{l+1}]\geq l$. Hence as above 
$[\beta]=[\beta_1]+\cdots+[\beta_{l+1}]\geq -l + [\beta_{l+1}]\geq 0$. 
\end{proof}


\noindent
{\sc Homogeneity}.
We return to a heuristic discussion.
Provided we include, like for (\ref{ao32}), $a$ into our considerations,
the original equation (\ref{ao22}) has a scaling symmetry: 
Considering for $s\in(0,\infty)$
the parabolic space-time rescaling $Sy=(sy_1,s^2y_2)$, we have for any exponent $\alpha$
\begin{align}\label{ao85}
\lefteqn{(u,\xi,a)\;\mbox{satisfies (\ref{ao22})}}\nonumber\\
&\Longrightarrow\quad
\big(s^{-\alpha} u(S\cdot),s^{2-\alpha}\xi(S\cdot),a(s^{\alpha}\cdot)\big)
=:(\tilde u,\tilde\xi,\tilde a)
\;\mbox{satisfies (\ref{ao22})}.
\end{align}
Suppose the scaling transformation
$\xi\mapsto\tilde\xi$ preserves the law, which for white
noise is the case with $\alpha-2=-\frac{D}{2}$, i.~e.~$\alpha=\frac{1}{2}$.
Since in view of Section \ref{sec:post}, the counter term only depends on the law,
it is natural to postulate, in line with that section, that 
the solution manifold of the renormalized problem inherits this invariance\footnote{
since this scale invariance in law is not consistent with
the mollification $\xi_\tau$ this discussion pertains to the limiting
solution manifold}.


It is also natural to postulate that the parameterization by the $p$'s (given a base point $x$)
is consistent with (\ref{ao85}) in the sense that $p$ transforms as $u$, i.~e.~
we have invariance under
\begin{align*}
(u,\xi,a,x,p)\;\mapsto\;
(\tilde u,\tilde\xi,\tilde a,\tilde x:=S^{-1}x,\tilde p:=s^{-\alpha} p(S\cdot)).
\end{align*}
We now appeal to the series expansion (\ref{ao83}), both as it stands and 
with $(x,y,u,\xi,a,p)$ replaced
by $(\tilde x,\tilde y:=S^{-1}y,\tilde u,\tilde\xi,\tilde a,\tilde p)$.
Because of $u(y)-u(x)$ 
$=s^\alpha(\tilde u(\tilde y)-\tilde u(\tilde x))$, we
obtain a relation between the two right-hand sides. 
It is natural to postulate that the coefficients $\{\Pi_{\cdot,\beta}\}_{\beta}$
are individually consistent with this invariance, leading to 
\begin{align}\label{ao86}
\Pi_{Sx\beta}[\xi](Sy)=s^{|\beta|}\Pi_{x\beta}[s^{2-\alpha}\xi(S\cdot)](y),
\end{align}
where the ``homogeneity'' $|\beta|$ of the multi-index $\beta$ is given by
\begin{align}\label{cw24}
|\beta|:=\alpha(1+[\beta])+|\beta|_p,
\end{align}
cf.~(\ref{cw14}) and (\ref{cw15}). We note that
\begin{align}\label{as27}
|e_{\bf n}|=|{\bf n}|
\end{align}
so that (\ref{cw24}) is consistent with (\ref{ao59}).


Appealing once more to the invariance in law of $\xi$ under (\ref{ao85}), we 
obtain from (\ref{ao86}) that 
the law of $s^{-|\beta|}\Pi_{Sx\,\beta}(Sy)$ coincides with the law of $\Pi_{x\beta}(y)$, in particular 
\begin{align}\label{ao87}
\mbox{the law of}\;s^{-|\beta|}\Pi_{Sx\,\beta}(Sy)\;\mbox{does not depend on $s\in(0,\infty)$}.
\end{align}
By the invariance of the (original) solution manifold under 
$(u,\xi)\mapsto(\tilde u:=u(\cdot+z),\tilde\xi:=\xi(\cdot+z))$, 
which by our assumption (\ref{ao30bis}) is passed
on to the renormalized solution manifold, it is natural to impose that the parameterization
is invariant under $(u,\xi,x,p)\mapsto(\tilde u,\tilde\xi,x+z,p(\cdot+z))$, 
and that the coefficients in (\ref{ao01}) are individually consistent with this invariance,
so that we likewise have 
\begin{align}\label{ao88}
\mbox{the law of}\;\Pi_{x+z\,\beta}(y+z)\;\mbox{does not depend on $z\in\mathbb{R}^2$}.
\end{align}
Specifying to $x=0$, the invariance (\ref{ao87}) implies
that $\mathbb{E}^\frac{1}{p}|\Pi_{0\beta}(y)|^p$ depends on $y$ only through $\frac{y}{|y|}$.
From the invariance (\ref{ao88}) we thus learn that
$\mathbb{E}^\frac{1}{p}|\Pi_{x\beta}(y)|^p$ depends on $x,y$ only through $\frac{y-x}{|y-x|}$.
Since $\frac{y-x}{|y-x|}$ has compact range, this suggest that
\begin{align*}
\mathbb{E}^\frac{1}{p}|\Pi_{x\beta}(y)|^p\lesssim |y-x|^{|\beta|},
\end{align*}
which is our main result, see \eqref{cw01} in the next section.


The scaling invariance \eqref{ao85} also connects to the notion of ``subcriticality'' 
which is often referred to in the realm of singular SPDEs. 
Loosely speaking, it means that by zooming in on small scales, 
the nonlinear term becomes negligible. 
Indeed, as can be seen from \eqref{ao85}, the rescaled nonlinearity $\tilde a$ 
converges to the constant $a(0)$ in the limit $s\downarrow 0$, i.~e.~
the SPDE (\ref{ao22})
turns into a linear one. This is true iff $\alpha>0$, and provides
the reason for restricting to $\alpha>0$ in the assumption of Theorem~\ref{thm}, 
which is the sub-critical regime for \eqref{ao22}. 


\section{The main result:\,Stochastic estimates of the centered model}\label{sec:main}

\noindent
The main result in \cite{LOTT} states that the objects introduced in an informal way 
in the previous subsection 
indeed can be rigorously defined, 
provided the noise $\xi$ is replaced by its mollified version $\xi_\tau$ as discussed at the end of Section~\ref{sec:Schauder}.
Moreover, the constants $c_\beta$ can be chosen in such a ($\tau$-dependent) way, 
that the centered model $\Pi_x$ satisfies stochastic estimates 
that are uniform in the mollification scale $\tau>0$. 

\begin{theorem}\label{thm}
Suppose the law of $\xi$ is invariant under \eqref{ao30bis};
suppose that it satisfies a spectral gap inequality \eqref{cw17} 
with exponent $\alpha\in(\max\{0,1-\frac{D}{4}\},1)\setminus\mathbb{Q}$. 


Then given $\tau>0$, there exists a deterministic $c\in\mathbb{R}[[\mathsf{z}_k]]$, 
and for every $x\in\mathbb{R}^2$,
a random\footnote{by this we mean a formal power series in $\mathsf{z}_k,\mathsf{z}_{\bf n}$
with values in the twice continuously differentiable space-time functions} 
$\Pi_{x}\in C^2[[\mathsf{z}_k,\mathsf{z}_{\bf n}]]$, 
and a random $\Pi_{x}^{-}\in C^0[[\mathsf{z}_k,\mathsf{z}_{\bf n}]]$ 
that are related by
\eqref{ao49} and
\begin{align}
(\partial_2-\partial_1^2)\Pi_{x\beta}&=\Pi_{x\beta}^{-}
\quad (\mbox{{\rm mod} polynomial of degree } \leq|\beta|-2),\label{cw04}
\end{align}
and that satisfy \eqref{ao59}, the population condition \eqref{cw03} and 
\begin{align}\label{cw10}
c_\beta=0\quad\mbox{unless}\quad|\beta|< 2.
\end{align}
Moreover, we have for $p<\infty$, $x,y\in\mathbb{R}^{2}$ and $t>0$ the estimates
\begin{align}
\mathbb{E}^\frac{1}{p}|\Pi_{x\beta}(y)|^p&\lesssim_{\beta,p}|y-x|^{|\beta|},\label{cw01}\\
\mathbb{E}^\frac{1}{p}|\Pi_{x\beta t}^{-}(y)|^p&\lesssim_{\beta,p}
(\sqrt[4]{t})^{\alpha-2}(\sqrt[4]{t}+|y-x|)^{|\beta|-\alpha}.\label{cw02}
\end{align}
\end{theorem}

The important feature is that the constants in (\ref{cw01}) and (\ref{cw02})
are uniform in $\tau\downarrow 0$.

 
We remark that we may pass from (\ref{cw02}) to (\ref{cw01}) by Lemma \ref{lem:int}.
Indeed, because of (\ref{cw03}) we may restrict to $\beta$ with $[\beta]\ge 0$.
In this case, by our assumption $\alpha\not\in\mathbb{Q}$, 
\begin{align}\label{as17}
[\beta]\ge0\quad\overset{\eqref{cw24}}{\Longrightarrow}\quad
|\beta|\not\in\mathbb{Z},
\end{align}
next to $|\beta|\ge\alpha$.
Hence we may indeed apply Lemma \ref{lem:int} with $\eta=|\beta|$ and
(\ref{cw02}) as input. The output yields a $\Pi_{x\beta}$ satisfying
(\ref{cw04}) and (\ref{cw01}).

\medskip

\noindent
{\sc Uniqueness and (implicit) BPHZ renormalization}. The construction
of $\Pi_x$ in \cite{LOTT} proceeds by an inductive algorithm in $\beta$.
The ordering\footnote{
this ordering coincides with the one chosen in \cite{LO} 
but it slightly differs from the one in \cite{LOTT}, 
which is imposed by the restricted triangularity of 
${\rm d}\Gamma^*$ in Section \ref{sec:Mall}; 
for simplicity we stick to (\ref{ks04})} 
on the multi-indices is provided by 
\begin{equation}\label{ks04}
|\beta|_{\prec}:=|\beta|+\lambda\beta(0)\quad\mbox{for fixed}\;\lambda\in(0,\alpha),
\end{equation}
and we will write $\gamma\prec\beta$ for $|\gamma|_\prec <|\beta|_\prec$.
As opposed to the ordering provided by the homogeneity, 
$\prec$ is coercive: For 
fixed $\beta$ there are only finitely many $\gamma$ with 
$\gamma\prec\beta$, see (\ref{as09}), which is important for the estimates.
Moreover, (\ref{ks04}), as opposed to the ordering by homogeneity, 
allows for the triangular structure:
\begin{equation}\label{ks05}
\Pi_{x\beta}^{-}-c_\beta\;\;\mbox{depends on}\;(\Pi_{x\gamma},c_\gamma)\;
\mbox{only through}\;\gamma\;\mbox{with}\;\gamma\prec\beta. 
\end{equation}
%
Indeed, by the component-wise \eqref{ao51}, we have to check that for $k\geq0$ 
and $e_k+\beta_1+\dots+\beta_{k+1}=\beta$ we have $\beta_1,\dots,\beta_{k+1}\prec\beta$, 
and that for $l\geq1$ and $\beta_1+\dots+\beta_{l+1}=\beta$ with $((D^{(\0)})^l)_{\beta_{l+1}}^\gamma$ we have $\beta_1,\dots,\beta_l,\gamma\prec\beta$. 
In the former case, and for $k\geq1$, we obtain from \eqref{cw24} that 
$|\beta_1|+\cdots+|\beta_{k+1}|=|\beta|$, and thus by $|\cdot|\geq\alpha>0$ that 
$|\beta_1|,\dots,|\beta_{k+1}|<|\beta|$. 
We conclude together with $\beta_1(0),\dots,\beta_{k+1}(0)\leq\beta(0)$.
In the case $k=0$, again by \eqref{cw24}, 
we have $|\beta_1|=|\beta|$, 
and we conclude by $1+\beta_1(0)=\beta(0)$. 
In the latter case, we use \eqref{ao19} and \eqref{cw24} to see 
$|\beta_1|+\cdots+|\beta_l|+|\gamma|=|\beta|$, 
which by $l\geq1$ implies as above $|\beta_1|,\dots,|\beta_l|,|\gamma|<|\beta|$. 
From \eqref{ao20} we also read off that $\gamma(0)\leq\beta_{l+1}(0)$, 
and we thus conclude by $\beta_1(0),\dots,\beta_{l+1}(0)\leq\beta(0)$. 


We now argue that within this induction, $(c,\Pi_x,\Pi_x^{-})$ is determined.
Indeed, the uniqueness statement of Lemma \ref{lem:int} implies that for given $\beta$,
$\Pi_{x\beta}$ is determined by $\Pi_{x\beta}^{-}$. 
According to (\ref{ks05}), $\Pi_{x\beta}^{-}-c_\beta$ is determined
by the previous steps.
Finally, we note that provided $|\beta|<2$, we have
\begin{align}\label{cw16}
|\mathbb{E}\Pi_{x\beta t}^{-}(x)|\le\mathbb{E}|\Pi_{x\beta t}^{-}(x)|
\overset{\eqref{cw02}}{\lesssim}(\sqrt[4]{t})^{|\beta|-2}
\overset{t\uparrow\infty}{\rightarrow}0,
\end{align}
so that $c_\beta$, because it is deterministic\footnote{and
independent of the base point $x$} may be recovered from 
$c_\beta=-\lim_{t\uparrow\infty}\mathbb{E}(\Pi_{x\beta}^{-}-c_\beta)_t(x)$.
Hence also $c_\beta$ is determined. 
Fixing the counter term by making an expectation\footnote{in our case
it is a space-time next to an ensemble average} vanish like in (\ref{cw16}) 
corresponds to what Hairer assimilates to a BPHZ renormalization. See
\cite[Theorem 6.18]{BHZ} for the form BPHZ renormalization takes
within regularity structures.

\medskip

\noindent
{\sc Mission accomplished.}
Returning to the end of Section \ref{sec:Schauder}, we may claim ``mission accomplished'':
\begin{itemize}
\item On the one hand, the form of the counter terms preserve a number of symmetries
of the original solution manifold: shift in $x$, reflection in $x_1$, shift in $u$,
and to some extend are guided by scaling in $x$.
\item On the other hand, in a term-by-term sense as encoded by (\ref{ao83}), the solution manifold
of the renormalized equation stays under control as $\tau\downarrow 0$, cf.~(\ref{cw01})
and (\ref{cw02}).
\end{itemize}
Moreover, the constants $c_\beta=c_\beta^{\tau}$ that determine the counter term
via (\ref{cw11}) are (canonically) determined by the large-scale part of
the estimate (\ref{cw02}).


As discussed in the introduction,
the connection between this term-by-term approach to the solution manifold
and the solution of an actual initial/boundary value problem is provided by
the second part of regularity structures.
This second part, a fixed point argument based
on a truncation of (\ref{ao83}) to a finite sum\footnote{by restricting to
homogeneities $|\beta|<2$; in our quasi-linear case,
the sum stays infinite w.~r.~t.~the $\mathsf{z}_0$-variable, but one has
analyticity in that variable since $1+\mathsf{z}_0$ plays the role
of a constant elliptic coefficient}, is not addressed
in these lecture notes.


\section{Malliavin derivative and Spectral gap (SG)}\label{sec:Mall}

\noindent
In view of the discussion at the end of the statement of Theorem \ref{thm}, the
main issue is the estimate (\ref{cw02}) of $\Pi_{x\beta}^{-}$. Indeed, 
its definition of (\ref{ao49}) still contains the singular
product $\Pi_x^k\partial_1^2\Pi_x$
{and the collection of
deterministic constants $c$
that diverge as the UV regularization fades away. Hence we seek a
relation between $\Pi_x^{-}$ and $\Pi_x$ that is more stable than (\ref{ao49});
in fact, it will be a relation between the {\it families} $\{\Pi_x^{-}\}_x$
and $\{\Pi_x\}_x$ based on symmetries under a change of the base point $x$.
This relation is formulated on the level of the derivative w.~r.~t.~the noise $\xi$,
also known as the Malliavin derivative. We start by motivating this approach.

\medskip

\noindent
{\sc Heuristic discussion of a stable relation} $\{\Pi_x\}_x\mapsto\{\Pi_x^{-}\}_x$.
Let $\delta$ denote the operation of taking the derivative of an object like $\Pi_{x\beta}(y)$,
which is a functional of $\xi$, in direction of an infinitesimal variation $\delta\xi$ of
the latter\footnote{in the Gaussian case, this would be an element of the Cameron-Martin space}.
Clearly, since $c_\beta$ is deterministic, we have $\delta c_\beta=0$. However, applying
$\delta$ to (a component of) (\ref{ao49}) does not eliminate $c$ because of the specific
way $c$ enters (\ref{ao49}), which is dictated by the fundamental symmetry (\ref{ao09}).
However, when evaluating (\ref{ao49}) at the base point $x$ itself and appealing to the built-in
\begin{equation}\label{ks02}
\Pi_x(x)=0, 
\end{equation}
see (\ref{ao83}) or (\ref{cw01}), 
it collapses to
\begin{align}\label{cw32}
\Pi_x^{-}(x)=\mathsf{z}_0\partial_1^2\Pi_x(x)-c+\xi_\tau(x)\mathsf{1}.
\end{align}
This isolates $c$ so that it can be eliminated by applying $\delta$:
\begin{align}\label{cw20}
\delta\Pi_x^{-}(x)=\mathsf{z}_0\partial_1^2\delta\Pi_x(x)+\delta\xi_\tau(x)\mathsf{1}.
\end{align}
Clearly, (\ref{cw20}) is impoverished in the sense that the active point coincides with
the base point.


Instead of attempting to modify the active point, the idea is to
modify the base point from $x$ to $y$. Such a change of base point, which will be
rigorously introduced in Section \ref{sec:Gamma}, amounts
to a change of coordinates in the heuristic representation (\ref{ao01}):
\begin{align}\label{cw21}
u=\left\{
\begin{array}{l}
u(x)+\sum_{\beta}\Pi_{x\beta}\mathsf{z}^\beta[a(\cdot+u(x)),p_x],\\
u(y)+\sum_{\beta}\Pi_{y\beta}\mathsf{z}^\beta[a(\cdot+u(y)),p_y],
\end{array}\right.
\end{align}
for some polynomials $p_x,p_y$ vanishing at the origin. The form in which the
$u$-shift appears in (\ref{cw21}) suggests that this change
of coordinates can be algebrized by an algebra endomorphism\footnote{in a first
reading, the star should be seen as mere notation; $\Gamma_{yx}^*$ is actually the
algebraic dual of a {\it linear} endomorphism $\Gamma_{yx}$ on the pre-dual space, 
see Lemma~\ref{LEM:def};
it is $\Gamma_{yx}$ that can be assimilated to the object denoted by the same symbol
in regularity structures; for a concise reference see \cite[Section 5.3]{LOT}} $\Gamma_{yx}^*$
of $\mathbb{R}[[\mathsf{z}_k,\mathsf{z}_{\bf n}]]$ with the properties
\begin{align}\label{cw27}
\Pi_y=\Gamma_{yx}^*\Pi_x+\Pi_y(x)\quad\mbox{and}\quad \Gamma_{yx}^*=\sum_{l\ge 0}
\frac{1}{l!}\Pi_y^l(x)(D^{({\bf 0})})^l\;\mbox{on}\;\mathbb{R}[[\mathsf{z}_k]],
\end{align}
see the discussion of finite $u$-shifts around (\ref{cw11}). 
Recall that an algebra endomorphism $\Gamma^*_{yx}$ is a linear map from $\mathbb{R}[[\mathsf{z}_k,\mathsf{z}_{\bf n}]]$ to $\mathbb{R}[[\mathsf{z}_k,\mathsf{z}_{\bf n}]]$ satisfying
\begin{equation}\label{mu}
\Gamma^*_{yx} \pi \pi' = (\Gamma^*_{yx} \pi)(\Gamma^*_{yx}\pi') 
\quad\textnormal{for }\pi,\pi'\in \mathbb{R}[[\mathsf{z}_k,\mathsf{z}_{\bf n}]].
\end{equation}
We claim that (\ref{cw27}) implies 
\begin{align}\label{cw28}
\Pi_y^{-}=\Gamma_{yx}^*\Pi_x^{-}.
\end{align} 
Indeed, applying $\Gamma_{yx}^*$ to definition (\ref{ao49}) we obtain by (\ref{mu})
\begin{align*}
\Gamma_{yx}^*\Pi_{x}^{-}
\hspace{-.5ex}=\hspace{-.5ex}
\sum_{k\ge 0}(\Gamma_{yx}^*\mathsf{z}_k)(\Gamma_{yx}^*\Pi_x)^k
\partial_1^2\Gamma_{yx}^*\Pi_x
\hspace{-.3ex}-\hspace{-.3ex}
\sum_{l\ge 0}\tfrac{1}{l!}(\Gamma_{yx}^*\Pi_x)^l\Gamma_{yx}^*(D^{({\bf 0})})^lc
+\xi_\tau\mathsf{1}.
\end{align*}
We substitute $\Gamma_{yx}^*\Pi_x$ according to the first item in (\ref{cw27}),
substitute $\Gamma_{yx}^*\mathsf{z}_k$ $=\sum_{l\ge 0}\tbinom{k+l}{k}\Pi_y^l(x)\mathsf{z}_{k+l}$
and $\Gamma_{yx}^*(D^{({\bf 0})})^lc$
according to the second item in (\ref{cw27}) and definition (\ref{ao13}),
and finally appeal to the binomial formula in both ensuing double sums to obtain 
(\ref{ao49}) with $x$ replaced by $y$, establishing (\ref{cw28}).


In view of the scaling (\ref{ao87}) and the transformation \eqref{cw27} we expect that 
the laws of $s^{|\beta|-|\gamma|}(\Gamma^*_{yx})_\beta^\gamma$ and of 
$(\Gamma^*_{SySx})_\beta^\gamma$ to be identical.
On the other hand, we expect $(\Gamma^*_{SySx})_\beta^\gamma$ to
converge to $(\Gamma^*_{00})_\beta^\gamma$ as $s\downarrow0$, and we expect
$\Gamma^*_{00}$ to be the identity. This suggests strict triangularity: 
\begin{align}\label{cw26}
(\Gamma^*_{yx}-{\rm id})_{\beta}^{\gamma}=0\quad\mbox{unless}\quad|\gamma|<|\beta|.
\end{align}


We claim that applying $\Gamma_{yx}^*$ to (\ref{cw20}), we obtain\footnote{
of course, the r.~h.~s.~term $\delta\Pi_y(x)$ is effectively absent due
to the derivative $\partial_1^2$}
\begin{align}\label{cw25}
\lefteqn{\delta\Pi_y^{-}(x)-(\delta\Gamma_{yx}^*)\Pi_x^{-}(x)}\nonumber\\
&=\sum_{k\ge 0}\mathsf{z}_k\Pi_y^k(x)\partial_1^2\big(\delta\Pi_y-\delta\Pi_y(x)
-(\delta\Gamma_{yx}^*)\Pi_x\big)(x)+\delta\xi_\tau(x)\mathsf{1}.
\end{align}
Since by (\ref{cw26}), $\delta\Gamma_{yx}^*$ is strictly triangular, (\ref{cw25})
provides an inductive way of determining $\{\Pi_x^{-}\}_x$ (up to expectation)
in terms of $\{\Pi_x\}_x$. Here comes the argument for (\ref{cw25}): Applying
$\Gamma_{yx}^*$ to the l.~h.~s.~of (\ref{cw20}) and using (\ref{cw28}) in
conjunction with Leibniz' rule w.~r.~t.~$\delta$,
we obtain the l.~h.~s.~of (\ref{cw25}). For the r.~h.~s.~we first use
the multiplicativity of $\Gamma_{yx}^*$; according to the second item in (\ref{cw27}) 
and (\ref{ao13}) we have 
\begin{align}\label{cw36}
\Gamma_{yx}^*\mathsf{z}_0=\sum_{l\ge 0}\Pi_y^l(x)\mathsf{z}_l.
\end{align}
To rewrite $\Gamma_{yx}^*\delta\Pi_x$, we apply
$\delta$ to the first identity in (\ref{cw27}). This establishes (\ref{cw25}).


We now argue that from an analytical point of view, (\ref{cw25}) is not quite adequate.
Clearly, the r.~h.~s.~of (\ref{cw25}) still contains a potentially singular product
of $\Pi_y^k$ and $\partial_1^2(\delta\Pi_y-\delta\Pi_y(x)$ 
$-(\delta\Gamma_{yx}^*)\Pi_x)$. Here, it is crucial that applying $\delta$ to 
$\Pi_y$, which is a multi-linear expression in $\xi$, means replacing one of the
instances of $\xi$ by $\delta\xi$. Now as we shall explain in the next subsection,
$\delta\xi$ gains\footnote{however on an $L^2$ instead of a uniform scale} 
$\frac{D}{2}$ orders  of regularity over $\xi$. However, since the other instances
of $\xi$ remain, the regularity of $\delta\Pi_y$ is not at face value better by
$\frac{D}{2}$ orders over $\Pi_y$, which is just H\"older continuous with exponent $\alpha$. 
Hence we can only expect that $\delta\Pi_y$ is locally, i.~e.~near a base point $x$, 
described -- ``modelled'' in the jargon of regularity structures -- to order
$\frac{D}{2}+\alpha$ in terms of $\Pi_x$. The Taylor-remainder-like expression 
$\delta\Pi_y-\delta\Pi_y(x)$ $-(\delta\Gamma_{yx}^*)\Pi_x$ has the potential of 
expressing this modeledness.
Hence the product of $\Pi_y^k$ and $\partial_1^2(\delta\Pi_y-\delta\Pi_y(x)$ 
$-(\delta\Gamma_{yx}^*)\Pi_x)$ has a chance of being well-defined 
provided $\alpha+(\frac{D}{2}+\alpha-2)>0$, which gives rise to the
lower bound assumption $\alpha>1-\frac{D}{4}$ in Theorem \ref{thm},
which reduces to\footnote{This is the analogy of rough path 
construction of fractional Brownian motion. For the case of fractional Brownian motion 
with Hurst parameter $H$, a rough path construction can be only implemented for any $H>\frac 14$ 
by increasing the number of iterated integrals. However, the stochastic analysis 
to construct the iterated integrals fails for fractional Brownian motion of 
Hurst parameter $H \le \frac 14$. See \cite[Theorem 2]{CQ}.}
$\alpha>\frac{1}{4}$ for our $D=3$.
Since $\frac{D}{2}+\alpha>1$, this only has
a chance of working provided every $\beta$-component of $(\delta\Gamma_{yx}^*)\Pi_x$ involves 
the affine function $\Pi_{xe_{(1,0)}}=(\cdot-x)_1$. However, this contradicts the
(strict) triangularity (\ref{cw26}) for $|\beta|\le 1$. Hence $\delta\Gamma_{yx}^*$ is
not rich enough to describe all components of $\delta\Pi_y$ to the desired order near $x$.


In view of the preceding discussion, we are forced to loosen the population
constraint (\ref{cw26}). To this purpose, we replace 
the directional Malliavin derivative $\delta\Gamma_{yx}^*$
by some ${\rm d}\Gamma_{yx}^*\in{\rm End}(\mathbb{R}[[\mathsf{z}_k,\mathsf{z}_{\bf n}]])$  
in order to achieve
\begin{align}\label{cw39}
\delta\Pi_y-\delta\Pi_y(x)-{\rm d}\Gamma_{yx}^*\Pi_x=O(|\cdot-x|^{\frac{D}{2}+\alpha}).
\end{align}
In order to preserve the identity (\ref{cw25}) in form of 
\begin{align}\label{cw29}
\lefteqn{\delta\Pi_y^{-}(x)-{\rm d}\Gamma_{yx}^*\Pi_x^{-}(x)}\nonumber\\
&=\sum_{k\ge 0}\mathsf{z}_k\Pi_y^k(x)\partial_1^2\big(\delta\Pi_y-\delta\Pi_y(x)
-{\rm d}\Gamma_{yx}^*\Pi_x\big)(x)+\delta\xi_\tau(x)\mathsf{1},
\end{align}
we need ${\rm d}\Gamma_{yx}^*$ to inherit the algebraic properties of $\delta\Gamma_{yx}^*$.
More precisely, we impose that ${\rm d}\Gamma_{yx}^*$ agrees with $\delta\Gamma_{yx}^*$
on the sub-algebra $\mathbb{R}[[\mathsf{z}_k]]$,
\begin{align}\label{cw30}
{\rm d}\Gamma_{yx}^*=\delta\Gamma_{yx}^*\;\mbox{on}\;\mathbb{R}[[\mathsf{z}_k]],
\end{align}
and that ${\rm d}\Gamma_{yx}^*$ 
is in the tangent space to the manifold of {\it algebra} morphisms in
$\Gamma_{yx}^*$, which means that for all $\pi,\pi'$ 
$\in\mathbb{R}[[\mathsf{z}_k,\mathsf{z}_{\bf n}]]$
\begin{align}\label{cw31}
{\rm d}\Gamma_{yx}^*\pi\pi'=({\rm d}\Gamma_{yx}^*\pi)(\Gamma_{yx}^*\pi')
+(\Gamma_{yx}^*\pi)({\rm d}\Gamma_{yx}^*\pi').
\end{align}
Here is the argument on how to pass from (\ref{cw30}) \& (\ref{cw31}) to (\ref{cw29}).
On the one hand, we apply $\delta$ to (\ref{ao49}) to the effect of
\begin{align}\label{cw33}
\delta\Pi_y^{-}(x)&=\sum_{k\ge 0}\mathsf{z}_k\delta\big(\Pi_y^{k}(x)\big)\partial_1^2\Pi_y(x)
+\sum_{k\ge 0}\mathsf{z}_k\Pi_y^k(x)\partial_1^2\delta\Pi_y(x)\nonumber\\
&-\sum_{l\ge 0}\frac{1}{l!}\delta\big(\Pi_y^{l}(x)\big)(D^{({\bf 0})})^lc
+\delta\xi_\tau(x)\mathsf{1}.
\end{align}
On the other hand, we apply ${\rm d}\Gamma_{yx}^*$ to (\ref{cw32}) to obtain by\footnote{which
also implies ${\rm d}\Gamma_{yx}^*\mathsf{1}=0$} (\ref{cw31})
\begin{align}\label{cw34}
\lefteqn{{\rm d}\Gamma_{yx}^*\Pi_x^{-}(x)}\nonumber\\
&=({\rm d}\Gamma_{yx}^*\mathsf{z}_0)\partial_1^2\Gamma_{yx}^*\Pi_x(x)
+(\Gamma_{yx}^*\mathsf{z}_0)\partial_1^2{\rm d}\Gamma_{yx}^*\Pi_x(x)
-{\rm d}\Gamma_{yx}^*c.
\end{align}
We now argue
that the first r.~h.~s.~term of (\ref{cw33}) is identical to the one in (\ref{cw34});
indeed, by the first item in (\ref{cw27}) we have $\partial_1^2\Gamma_{yx}^*\Pi_x$
$=\partial_1^2\Pi_y$. On the other hand, by (\ref{cw30}) and the second item
in (\ref{cw27}) we have
\begin{align}\label{cw35}
{\rm d}\Gamma_{yx}^*=\sum_{l\ge 0}\frac{1}{l!}\delta\big(\Pi_y^l(x)\big)(D^{({\rm 0})})^l
\quad\mbox{on}\;\mathbb{R}[[\mathsf{z}_k]].
\end{align}
so that by (\ref{ao13}) ${\rm d}\Gamma_{yx}^*\mathsf{z}_0$
$=\sum_{k\ge 0}\delta(\Pi_y^k(x))\mathsf{z}_k$. Identity (\ref{cw35}) also implies
that the third r.~h.~s.~terms of (\ref{cw33}) and (\ref{cw34}) are identical.
The second r.~h.~s.~terms of (\ref{cw33}) and (\ref{cw34}) combine as desired by (\ref{cw36}).
This establishes (\ref{cw29}). In order to use (\ref{cw29}) inductively to define
-- or rather estimate -- $\{\Pi_x^{-}\}_x$, \cite{LOTT} had to come up with an
ordering on multi-indices $\beta$ with respect to which ${\rm d}\Gamma_{yx}^*$ is
strictly triangular, leading to a modification of (\ref{ks04}). 

Incidentally, the point of view adopted in \cite{BOTT} allows for 
a more geometric interpretation of ${\rm d}\Gamma^*_{yx}$: 
The linear combination $\delta\Pi_y(x)+{\rm d}\Gamma^*_{yx}\Pi_x$ is actually 
an element of the tangent space of the solution manifold of \eqref{ao27}; 
The Malliavin derivative $\delta\Pi_y$ lies approximately in the aforementioned tangent space, 
which is expressed by \eqref{cw39}. 
For more details we invite the reader to have a look at \cite{BOTT}. 

\medskip
\noindent
{\sc Definition of the Malliavin derivative and SG}.
We have seen that the Malliavin derivative, which we now shall rigorously
define, allows to give a more
robust relation between $\Pi_x$ and $\Pi_x^{-}$. Via the SG inequality, which will be introduced
here, the control of the Malliavin derivative of a random variable $F$ yields control of the variance
of $F$.
Consider the Hilbert norm on (a subspace of) the space of Schwartz distributions\footnote{we
denote the argument by $\delta\xi$ since we think of it as an infinitesimal
perturbation.}
\begin{align}
\|\delta\xi\|^2=\int_{\mathbb{R}^2}dx\big((\partial_1^4-\partial_2^2)
^{\frac{1}{4}(\alpha-\frac{1}{2})}\delta\xi\big)^2
=\int_{\mathbb{R}^2}dq\big||q|^{(\alpha-\frac{1}{2})}\mathcal{F}\delta\xi\big|^2.
\label{Hilbert}
\end{align}
Note that we encounter again $A^*A=(-\partial_2-\partial_1^2)(\partial_2-\partial_1^2)$
with Fourier symbol $|q|^4=q_1^4+q_2^2$, see (\ref{ao79}).
Hence this is one of the equivalent ways of defining
the homogeneous $L^2(\mathbb{R}^2)$-based Sobolev norm of fractional
order $\alpha-\frac{1}{2}$, however of parabolic scaling, which we nevertheless still
denote by $H:=\dot H^{\alpha-\frac{1}{2}}(\mathbb{R}^2)$.


We now consider ``cylindrical'' (nonlinear) functionals $F$ on the space 
${\mathcal S}'(\mathbb{R}^2)$
of Schwartz distributions, by which one means that for some $N\in\mathbb{N}$,
$F$ is of the form
\begin{align}\label{as30}
&F[\xi]=f\big((\xi,\zeta_1),\cdots,(\xi,\zeta_N)\big)\quad\mbox{with}\nonumber\\
&f\in C^\infty(\mathbb{R}^N)\;\mbox{and}\;\zeta_1,\cdots,\zeta_N\in{\mathcal S}(\mathbb{R}^2),
\end{align}
where we recall that
$(\xi,\zeta_n)$ denotes the natural pairing between $\xi\in{\mathcal S}'(\mathbb{R}^2)$
and a Schwartz function $\zeta_n\in{\mathcal S}(\mathbb{R}^2)$.
Clearly, those function(al)s $F$ are Fr\'echet differentiable with
\begin{align}\label{as31}
{\rm d}F[\xi].\delta\xi&=\lim_{s\downarrow 0}\frac{1}{s}(F[\xi+s\delta\xi]-F[\xi])\nonumber\\
&=\sum_{n=1}^N\partial_nf\big((\xi,\zeta_1),\cdots,(\xi,\zeta_N)\big)\,(\delta\xi,\zeta_n)
=(\delta\xi,\frac{\partial F}{\partial\xi}[\xi]),
\end{align}
where $\frac{\partial F}{\partial\xi}[\xi]\in{\mathcal S}(\mathbb{R}^2)$ is defined through
\begin{align*}
\frac{\partial F}{\partial\xi}[\xi]
=\sum_{n=1}^N\partial_nf\big((\xi,\zeta_1),\cdots,(\xi,\zeta_N)\big)\,\zeta_n.
\end{align*}
We will monitor the dual norm
\begin{align}\label{as32}
\|\frac{\partial F}{\partial\xi}[\xi]\|_*:=\sup_{\delta\xi}
\frac{(\delta\xi,\frac{\partial F}{\partial\xi}[\xi])}{\|\delta\xi\|}=
\|\frac{\partial F}{\partial\xi}[\xi]\|_{\dot H^{\frac{1}{2}-\alpha}(\mathbb{R}^2)}.
\end{align}

\begin{definition}
An ensemble $\mathbb{E}$ of Schwartz distributions\footnote{ 
It does not have to be a Gaussian ensemble.} is said to satisfy a SG inequality
provided for all cylindrical $F$ with $\mathbb{E}|F|<\infty$
\begin{align}\label{cw17}
\mathbb{E}(F-\mathbb{E}F)^2\le\mathbb{E}\|\frac{\partial F}{\partial\xi}\|_*^2.
\end{align}
\end{definition}

Note that the l.~h.~s.~of (\ref{cw17}) is the variance of $F$.
Inequality (\ref{cw17}) amounts to an $L^2$-based Poincar\'e inequality with
mean value zero on the (infinite-dimensional) space of all $\xi$'s.
By a (parabolic) rescaling of $x$, we may w.~l.~o.~g.~assume that the constant
in (\ref{cw17}) is unity.
Implicitly, we also include closability of the linear operator
\begin{align}\label{cw18}
\mbox{cylindrical function}\;F\;\mapsto\;\frac{\partial F}{\partial\xi}
\in\{\mbox{cylindrical functions}\}\otimes{\mathcal S}(\mathbb{R}^2).
\end{align}
This means that the closure of the graph of (\ref{cw18}) w.~r.~t.~the topology
of $\mathbb{L}^2$ and $\mathbb{L}^2(H^*)$ is still a graph. This allows to
extend the Fr\'echet derivative (\ref{cw18}) to the Malliavin derivative
\begin{align*}
\mathbb{L}^2\supset{\mathcal D}(\frac{\partial}{\partial\xi})\ni F
\;\mapsto\;\frac{\partial F}{\partial\xi}\in\mathbb{L}^2(H^*).
\end{align*}
By the chain rule, we may post-process (\ref{cw17}) to its $\mathbb{L}^p$-version
\begin{align}
\mathbb{E}^\frac{1}{p}|F-\mathbb{E}F|^p\lesssim_p\mathbb{E}^\frac{1}{p}
\|\frac{\partial F}{\partial\xi}\|_*^p,
\label{cw18a}
\end{align}
which is the form we use it in. A concise proof how to obtain \eqref{cw18a} from 
\eqref{cw17} can be found in \cite[Step 2 in the proof of Lemma 3.1]{JO}.


The obvious examples are Gaussian ensembles of Schwartz distributions with
\begin{align}
\|\cdot\|\le\mbox{Cameron-Martin norm},
\label{CM}
\end{align}
where the norm $\| \cdot \|$ means the Hilbert norm defined in \eqref{Hilbert},
e.~g.~
\begin{align*}
\begin{array}{lrcl}
\mbox{white noise}&-\frac{D}{2}&=\alpha-2&\Longrightarrow\;\alpha=\frac{1}{2},\\[1ex]
\mbox{free field}&1-\frac{D}{2}&=\alpha-2&\Longrightarrow\;\alpha=\frac{3}{2}.
\end{array}
\end{align*}
In other words, the SG inequality \eqref{cw17} holds with Gaussian ensembles satisfying \eqref{CM}, see \cite[Theorem 5.5.11]{bogachev}.


For the reader's convenience, we sketch the simplest application of SG from 
\cite[Section 4.3]{LOTT}, namely (\ref{cw02}) for $\beta=0$.
To this aim we apply (\ref{cw18a}) to $F:=(\xi,\psi_t(y-\cdot))=\Pi_{x0t}^{-}(y)$,
which is of the form of (\ref{as30}), so that according to (\ref{as31})
its Malliavin derivative is given by $\frac{\partial F}{\partial\xi}=\psi_t(y-\cdot)$.
In view of (\ref{as32}), and then appealing to (\ref{ao37})
in conjunction with the translation invariance and scaling of the Sobolev norm we have
\begin{align*}
\|\frac{\partial F}{\partial\xi}\|_*
=\|\psi_t(y-\cdot)\|_{\dot H^{\frac{1}{2}-\alpha}(\mathbb{R}^2)}
=(\sqrt[4]{t})^{-\frac{D}{2}-\frac{1}{2}+\alpha}
\|\psi_{t=1}\|_{\dot H^{\frac{1}{2}-\alpha}(\mathbb{R}^2)}.
\end{align*}
Noting that the exponent is $\alpha-2$ and that $\psi_{t=1}$ is a (deterministic) Schwartz function
we obtain from (\ref{cw18a})
\begin{align*}
\mathbb{E}^\frac{1}{p}|\Pi_{x0t}^{-}(y)|^p\lesssim(\sqrt[4]{t})^{\alpha-2}.
\end{align*}
In view of $|0|=\alpha$, this amounts to the desired (\ref{cw02}) for $\beta=0$.


We also remark that SG naturally complements the BPHZ-choice of renormalization,
see Section \ref{sec:main}:
\begin{itemize}
\item The choice of $c_\beta$ takes care of the mean $\mathbb{E}\Pi_{x\beta t}^{-}(y)$, while
\item SG takes care of the variance of $\Pi_{x\beta t}^{-}(y)$.
\end{itemize}
Hence the main task in \cite{LOTT} is the estimate of
$\mathbb{E}^\frac{1}{p}\|\frac{\partial F}{\partial\xi}\|_*^p$, where $F:=\Pi_{x\beta t}^{-}(y)$,
which we tackle by duality through estimating the directional derivative
\begin{align*}
\delta F:=(\delta\xi,\frac{\partial F}{\partial\xi})\quad\mbox{given control of}\;
\mathbb{E}^\frac{1}{q}\|\delta\xi\|^q.
\end{align*}
The inductive estimate is based on (\ref{cw29}). 


Philosophically speaking, our approach is analytic rather than combinatorial:

\begin{center}
\small
\begin{tabular}{r|l|l}
 & \multicolumn{1}{c|}{analytic} & \multicolumn{1}{c}{combinatorial} \\
 \hline
index set: & derivatives w.r.t. $a$ and $p$ & Picard iteration \\
 & $\leadsto$ multi-indices on $k\geq0$, $\n\neq\0$ & $\leadsto$ trees with decorations \\

 \hline
Ass. on $\xi$: & spectral gap inequality & cumulant bounds \\
 & Malliavin derivatives w.r.t. $\xi$ & trees with paired nodes \\
 & $\leadsto$ estimates on $\mathbb{E}\|\frac{\partial}{\partial\xi}\Pi_{x\beta\,t}^{-}(y)\|_*^2$ & $\leadsto$ Feynman diagrams 
\end{tabular}
\end{center}
%
For us, all combinatorics are contained in Leibniz' rule. We also point out that our approach 
may be called ``top-down" rather than bottom-up in the sense that we postulate 
the conditions (space-time translation, spatial reflection, shift-covariance, etc) 
on the counter term $h$ from the beginning.


A closing remark for experts in QFT:
The absence of $c$ in (\ref{cw29}) means that our approach does not suffer from
the well-known difficulty of ``overlapping sub-divergences'' in Quantum Field Theory,
which is also an issue in \cite{ch16}. Our inductive approach has similarities
with the one of Epstein-Glaser, see \cite[Section 3.1]{scharf}.


\section{The structure group and the re-expansion map}\label{sec:Gamma}

\noindent
In this section we construct the endomorphism $\Gamma^*_{yx}$ of the algebra
$\mathbb{R}[[\mathsf{z}_k,\mathsf{z}_{\bf n}]]$ that satisfies \eqref{cw27} 
for given $\Pi_x$ and $\Pi_y$. In \cite{LOTT}, the constructions (and estimates)
of $\Gamma_{yx}^*$ and $\Pi_x$ are actually intertwined, however the proof
of Lemma \ref{choice} has the same elements as \cite[Section 5.3]{LOTT}. 
In line with regularity structures it is convenient to adopt 
a more abstract point of view:
We start by introducing what can be assimilated
to Hairer's structure group $\mathsf{G}$, which here is a subgroup of the automorphism group
of the linear space $\mathbb{R}[\mathsf{z}_k,\mathsf{z}_{\bf n}]$, where 
$\mathbb{R}[\mathsf{z}_k,\mathsf{z}_{\bf n}]$ now plays the role
of the\footnote{canonical w.~r.~t.~the monomial basis} 
(algebraic) pre-dual of $\mathbb{R}[[\mathsf{z}_k,\mathsf{z}_{\bf n}]]$;
$\Gamma^*_{yx}$ will be the transpose of a $\Gamma_{yx}\in\mathsf{G}$.
The elements $\Gamma\in\mathsf{G}$ are parameterized by 
$\{\pi^{({\bf n})}\}_{\bf n}$ $\subset\mathbb{R}[[\mathsf{z}_k,\mathsf{z}_{\bf n}]]$,
see Lemma~\ref{LEM:def}; the group property will be established in 
Lemma~\ref{lem:group}. In Lemma \ref{choice} we inductively choose
$\{\pi^{({\bf n})}_{yx}\}_{\bf n}$ such that the associated $\Gamma_{yx}$}
satisfies \eqref{cw27}. 
For a discussion of the Hopf- and Lie-algebraic structure underlying $\mathsf{G}$
we refer to \cite{LOT}. As opposed to \cite{LOT} and \cite{LO}, we will
capitalize on $\alpha<1$, which simplifies several arguments.

\begin{lemma}\label{LEM:def}
Given\footnote{which here as opposed to earlier includes the additional (dummy) 
index ${\bf n}={\bf 0}$ we first encountered in (\ref{ao06})} $\{\pi^{({\bf n})}\}_{\bf n}$ $\subset\mathbb{R}[[\mathsf{z}_k,\mathsf{z}_{\bf n}]]$ 
satisfying
\begin{align}
\pi^{({\bf n})}_\beta&=0\quad\mbox{unless}\quad|{\bf n}|<|\beta|,\label{def1}
\end{align}
there exists a unique linear endomorphism $\Gamma$ of 
$\mathbb{R}[\mathsf{z}_k,\mathsf{z}_{\bf n}]$
such that $\Gamma^*$ is an algebra endomorphism\footnote{i.~e.~$\Gamma^*\pi\pi'$
$=(\Gamma^*\pi)(\Gamma^*\pi')$ and $\Gamma^*\mathsf{1}=\mathsf{1}$ hold} of
$\mathbb{R}[[\mathsf{z}_k,\mathsf{z}_{\bf n}]]$ that satisfies
\begin{align}
\Gamma^*\mathsf{z}_k&=\sum_{l\ge 0}\frac{1}{l!}\,(\pi^{({\bf 0})})^l (D^{(\0)})^l \z_k
\overset{\eqref{ao13}}{=}\sum_{l\ge 0}\tbinom{k+l}{k}
(\pi^{({\bf 0})})^l\mathsf{z}_{k+l},\label{zk}\\
\Gamma^*\mathsf{z}_{\bf n}&=\mathsf{z}_{\bf n}+\pi^{({\bf n})}.\label{zn}
\end{align}
In addition\footnote{we recall that $\prec$ is defined in \eqref{ks04}}, 
\begin{equation}\label{ks06}
(\Gamma^*-{\rm id})_{\beta}^{\gamma}=0\quad\mbox{unless}\quad|\gamma|<|\beta|
\;\;\mbox{and}\;\;\gamma\prec\beta.
\end{equation}
\end{lemma}

We remark that the algebra endomorphism property, the mapping
property (\ref{zk}), and the first triangularity in (\ref{ks06}) mimic desired
properties of $\Gamma_{yx}^*$, namely (\ref{mu}), the second item of (\ref{cw27}), 
and (\ref{cw26}), respectively.


\begin{proof}[Proof of Lemma \ref{LEM:def}]
We recall that the matrix representation $\{\Gamma_\gamma^\beta\}_{\beta,\gamma}$
of a linear endomorphism $\Gamma$ of $\mathbb{R}[\z_k,\z_\n]$
w.~r.~t.~the monomial basis $\{\z^\beta\}_\beta$ is given by
\begin{equation}\label{as01}
\Gamma \z^\beta = \sum_\gamma \Gamma_\gamma^\beta \z^\gamma.
\end{equation}
The algebraic dual $\Gamma^*$, as a linear endomorphism of $\mathbb{R}[[\z_k,\z_\n]]$, 
is given by\footnote{note that the sum is effectively finite,
since there are only finitely many $\gamma$ such that $\Gamma_\gamma^\beta\neq0$
since the monomial basis is an algebraic basis} 
\begin{align*}
(\Gamma^*\pi)_\beta = \sum_\gamma (\Gamma^*)_\beta^\gamma \pi_\gamma\quad\mbox{where}\quad
(\Gamma^*)_\beta^\gamma:=(\Gamma^*\mathsf{z}^\gamma)_\beta=\Gamma_\gamma^\beta.
\end{align*}
Such a $\Gamma^*$ is an algebra endomorphism if and only if
\begin{align}
(\Gamma^*)_\beta^\gamma=\sum_{\beta_1+\cdots+\beta_k=\beta}
(\Gamma^*)_{\beta_1}^{\gamma_1}\cdots(\Gamma^*)_{\beta_k}^{\gamma_k}\quad
\mbox{for}\quad\gamma=\gamma_1+\cdots+\gamma_k.
\label{decomp1}
\end{align}
This includes $\Gamma^*\mathsf{1}=\mathsf{1}$ in form of
\begin{align}\label{as28}
(\Gamma^*)_\beta^0=\delta_\beta^0
\end{align}
Since any multi-index $\gamma\not=0$ can be written as the sum of $\gamma_j$'s of length one,
we learn that an endomorphism $\Gamma$ of $\mathbb{R}[\mathsf{z}_k,\mathsf{z}_{\bf n}]$ with
multiplicative $\Gamma^*$ is determined by how $\Gamma^*$ acts on the
coordinates $\{\mathsf{z}_k\}_{k\ge 0}$ and $\{\mathsf{z}_{\bf n}\}_{{\bf n}\not={\bf 0}}$. 
This establishes the uniqueness statement.


For the existence, we need to establish that the numbers 
$\{(\Gamma^*)_\beta^\gamma\}_{\beta,\gamma}$ defined through
(\ref{zk}) \& (\ref{zn}) in form of
\begin{align}
(\Gamma^*)_\beta^{e_k}-\delta_{\beta}^{e_k}&=\sum_{l\ge 1}\tbinom{k+l}{k}
\sum_{e_{k+l}+\beta_1+\cdots+\beta_l=\beta}
\pi^{({\bf 0})}_{\beta_1}\cdots\pi^{({\bf 0})}_{\beta_l},\label{as02}\\
(\Gamma^*)_\beta^{e_{\bf n}}-\delta_{\beta}^{e_{\bf n}}&=\pi_{\beta}^{(\bf n)}\label{as03}
\end{align}
and extended by (\ref{decomp1}) \& (\ref{as28}) to all $\gamma$ satisfy (for fixed $\beta$)
\begin{align}
\#\{ \gamma \, | \, (\Gamma^*)_\beta^\gamma \neq 0 \}<\infty.
\label{G1}
\end{align}
Indeed, this finiteness condition allows
to define $\Gamma$ via (\ref{as01}) with $\Gamma_{\gamma}^\beta$ $:=(\Gamma^*)_\beta^\gamma$. 
Since thanks to (\ref{as29}) below in conjunction with $0<\lambda,\alpha<1$
the ordering $\prec$ is coercive, by which we mean
\begin{align}\label{as09}
\#\{ \gamma \, | \, \gamma \prec\beta \}<\infty,
\end{align}
(\ref{G1}) follows once we establish the second strict triangularity in (\ref{ks06}).


Hence, it remains to establish (\ref{ks06}) in form of
\begin{align}\label{as04}
(\Gamma^*)_\beta^\gamma-\delta_\beta^\gamma=0\quad\mbox{unless}\quad
|\gamma|_{\prec}<|\beta|_{\prec}\quad\mbox{and}\quad|\gamma|<|\beta|
\end{align}
for the numbers 
$\{(\Gamma^*)_\beta^\gamma\}_{\beta,\gamma}$ defined through (\ref{as02}) \& (\ref{as03})
and then extended by (\ref{decomp1}). For this purpose, we note that by definition
(\ref{cw24}) in form of
\begin{align}\label{as29}
|\beta|-\alpha=\alpha\sum_{k\ge 0}k\beta(k)+\sum_{{\bf n}\not={\bf 0}}(|{\bf n}|-\alpha)\beta({\bf n})
\end{align}
and since $\alpha\le 1\le|{\bf n}|$,
\begin{align}\label{as05}
|\cdot|-\alpha\ge 0\quad\mbox{is additive}
\quad\overset{\eqref{ks04}}{\Longrightarrow}\quad\mbox{same for}\;|\cdot|_{\prec}-\alpha.
\end{align}
We first restrict to $\gamma$'s of length one in (\ref{as04}),
and distinguish the cases $\gamma=e_{\bf n}$ and $\gamma=e_k$.
Since by (\ref{cw24}) and (\ref{ks04}) we have $|e_{\bf n}|_{\prec}$
$=|e_{\bf n}|$ $=|{\bf n}|$ and $|\beta|\le|\beta|_{\prec}$, the former case
follows directly via (\ref{as03}) from assumption (\ref{def1}). We now turn to the latter case
of $\gamma=e_k$ and to (\ref{as02}). There is a contribution to the r.~h.~s.~sum
only when there exists an $l\ge 1$ and a decomposition $\beta=e_{k+l}+\beta_1+\cdots+\beta_{l}$;
this implies 
\begin{align*}
|\beta|\ge|e_{k+l}|\overset{\eqref{cw24}}{=}|e_k|+\alpha l\ge|e_k|+\alpha
\quad\overset{\eqref{ks04}}{\Longrightarrow}\quad
|\beta|_{\prec}\ge|e_k|_{\prec}+(\alpha-\lambda),
\end{align*}
which yields the desired (\ref{as04}) because of $\alpha>\lambda,0$.


Finally, we need to upgrade (\ref{as04}) from $\gamma$'s of length one
to those of arbitrary length, which we do by induction in the length. The base case of zero
length, i.~e.~of $\gamma=0$, is dealt with in (\ref{as28}).
We carry out the induction step with help of (\ref{decomp1}),
writing a multi-index $\gamma=\gamma'+\gamma''$ with $\gamma',\gamma''$ of smaller length:
\begin{align}\label{as07}
(\Gamma^*)_\beta^\gamma=\sum_{\beta'+\beta''=\beta}
(\Gamma^*)_{\beta' }^{\gamma' }(\Gamma^*)_{\beta''}^{\gamma''}.
\end{align}
We learn from the induction-hypothesis version of (\ref{as04}) that the summand vanishes unless
\begin{align*}
&|\gamma'|+|\gamma''|<|\beta'|+|\beta''|
\;\mbox{and}\;|\gamma'|_{\prec}+|\gamma''|_{\prec}<|\beta'|_{\prec}+|\beta''|_{\prec}
\nonumber\\
&\mbox{or}\quad\gamma'=\beta'\;\mbox{and}\;\gamma''=\beta'';
\end{align*}
in the latter case the summand is equal to $1$.
By (\ref{as05}), the first alternative implies
$|\gamma|<|\beta|$ and $|\gamma|_{\prec}<|\beta|_{\prec}$.
The second alternative implies $\gamma=\beta$
and then holds for exactly one summand to the desired effect of
$(\Gamma^*)_\beta^\gamma=1$.
\end{proof}

The two triangular properties (\ref{ks06}) from Lemma \ref{LEM:def}
allow us to establish the group property.
Furthermore, a triangular dependence (\ref{tri2}) of $\Gamma^*$ on $\pi^{({\bf n})}$ 
will play a crucial role when inductively constructing
$\pi^{({\bf n})}_{yx}$ in Lemma~\ref{choice}.

\begin{lemma}\label{lem:group}
The set $\mathsf{G}$ of all $\Gamma$ as in Lemma~\ref{LEM:def}
defines a subgroup of the automorphism group of $\mathbb{R}[\mathsf{z}_k,\mathsf{z}_{\bf n}]$.
Moreover,
\begin{align}
\mbox{for $[\gamma]\ge 0$,}\quad
(\Gamma^*)_\beta^\gamma\;\;\mbox{is independent of}\;\;\pi^{({\bf n})}_{\beta'} 
\quad\mbox{unless}\quad\beta'\prec\beta. \label{tri2}
\end{align}
\end{lemma}

\begin{remark}\label{rem:1}
The group $\mathsf{G}$ is larger than the one constructed in \cite{LOT},
since 1) we do not require that $\pi^{({\bf n})}_\beta=0$
unless $\beta$ satisfies \eqref{cw03}, and 2) we do not 
specify the space-time shift structure
of the $(\beta=e_{\bf m})$-components of $\pi^{({\bf n})}_\beta$
as in \cite[Proposition~5.1]{LOT}. Both conditions however
are satisfied for our construction 
of $\pi_{yx\beta}^{({\bf n})}$, see \eqref{as13} and \eqref{as25}.
\end{remark}


\begin{proof}[Proof of Lemma~\ref{lem:group}]
We first argue that for $\Gamma,\Gamma'\in\mathsf{G}$ we have $\Gamma'\Gamma\in\mathsf{G}$.
More precisely, if $\Gamma$ and $\Gamma'$ are associated to
$\{\pi^{({\bf n})}\}_{\bf n}$ and $\{\pi'^{({\bf n})}\}_{\bf n}$ by Lemma \ref{LEM:def}, 
respectively, we consider
\begin{align}
\widetilde{\pi}^{({\bf n})}:=\pi^{({\bf n})}+\Gamma^* \pi'^{({\bf n})}.
\label{pitilde}
\end{align}
We note that by triangularity (\ref{ks06}) of $\Gamma^*$ w.~r.~t.~$|\cdot|$, 
the population property (\ref{def1}) propagates from $\pi^{({\bf n})}$, $\pi'^{({\bf n})}$
to $\widetilde{\pi}^{({\bf n})}$.
Let $\widetilde{\Gamma}\in\mathsf{G}$ be associated to $\{\widetilde{\pi}^{({\bf n})}\}_{\bf n}$;
we claim that $\Gamma'\Gamma=\widetilde{\Gamma}$.


To this purpose, we note that $(\Gamma'\Gamma)^*$ $=\Gamma^*{\Gamma'}^*$
is an algebra morphism, like $\widetilde{\Gamma}^*$ is. Hence by the uniqueness statement
of Lemma \ref{LEM:def}, it is sufficient to check that $\Gamma^*{\Gamma'}^*$ and
$\widetilde{\Gamma}^*$ agree on the two sets of coordinates $\{\mathsf{z}_k\}_k$ 
and $\{\mathsf{z}_{\bf n}\}_{\bf n}$. On the latter this is easy:
\begin{align*}
{\widetilde \Gamma}^*\mathsf{z}_{{\bf n}}
&\overset{\eqref{zn}}{=}\mathsf{z}_{{\bf n}}+\widetilde{\pi}^{(\bf n)}
\overset{\eqref{pitilde}}{=}\mathsf{z}_{\bf n}+\pi^{({\bf n})}+\Gamma^* \pi'^{({\bf n})}
\overset{\eqref{zn}}{=}\Gamma^*(\mathsf{z}_{\bf n}+\pi'^{({\bf n})}) \\
&\overset{\eqref{zn}}{=}\Gamma^* \Gamma'^* \mathsf{z}_{{\bf n}}.
\end{align*}
We now turn to the $\mathsf{z}_k$'s, 
showing that the algebra endomorphisms $\Gamma^*{\Gamma'}^*$ and
$\widetilde{\Gamma}^*$ agree on the sub-algebra $\mathbb{R}[\mathsf{z}_k]$
$\subset\mathbb{R}[[\mathsf{z}_k,\mathsf{z}_{\bf n}]]$; by
multiplicativity of $\Gamma^*$ we have according to (\ref{zk}) for $\Gamma'$
\begin{align*}
\Gamma^*{\Gamma'}^*=\sum_{l'\ge 0}\frac{1}{l'!}(\Gamma^*{\pi'}^{({\bf 0})})^{l'}
\Gamma^*(D^{({\bf 0})})^{l'}\quad\mbox{on}\;\mathbb{R}[\mathsf{z}_k].
\end{align*}
Since $D^{({\bf 0})}$ preserves $\mathbb{R}[\mathsf{z}_k]$, we may apply
(\ref{zk}) for $\Gamma$ and obtain by the binomial formula:
\begin{align*}
\Gamma^*{\Gamma'}^*
&=\sum_{l'\ge 0}\frac{1}{l'!}(\Gamma^*{\pi'}^{({\bf 0})})^{l'}
\sum_{l\ge 0}\frac{1}{l!}(\pi^{({\bf 0})})^l(D^{({\bf 0})})^{l'+l}\nonumber\\
&
\overset{\eqref{pitilde}}{=}\sum_{\tilde l\ge0}\frac{1}{\tilde l!}
(\widetilde{\pi}^{({\bf 0})})^{\tilde l}(D^{({\bf 0})})^{\tilde l}
\quad\mbox{on}\;\mathbb{R}[\mathsf{z}_k],
\end{align*}
which according to (\ref{zk}) agrees with $\widetilde{\Gamma}^*$.


We come to the inverse of a $\Gamma\in\mathsf{G}$ associated to
$\{\pi^{({\bf n})}\}_{\bf n}$.
By the strict triangularity (\ref{ks06}) 
w.~r.~t.~the coercive $\prec$, cf.~(\ref{as09}), there exists $\tilde\pi^{({\bf n})}\in
\mathbb{R}[[\mathsf{z}_k,\mathsf{z}_\n]]$ such that
\begin{align}\label{as10}
\Gamma^*\tilde\pi^{({\bf n})}=-\pi^{({\bf n})}.
\end{align}
We now argue by induction in $\beta$ w.~r.~t.~$\prec$ that $\tilde\pi^{({\bf n})}$
satisfies (\ref{def1}). For this, we spell (\ref{as10}) out as
\begin{align*}
\tilde\pi^{({\bf n})}_\beta+\sum_{\gamma}(\Gamma^*-{\rm id})_\beta^\gamma
\tilde\pi^{({\bf n})}_\gamma=-\pi^{({\bf n})}_\beta.
\end{align*}
If $|\beta|\le|{\bf n}|$, the r.~h.~s.~vanishes by (\ref{def1}), and
by (\ref{ks06}) the sum over $\gamma$ restricts to
$|\gamma|\le|\beta|\le|{\bf n}|$, and to $\gamma\prec\beta$, 
so that the summand vanishes by induction hypothesis.
Thus also $\tilde\pi^{({\bf n})}_\beta$ vanishes.


This allows us to argue that $\tilde\Gamma\in\mathsf{G}$ associated to 
$\{\tilde\pi^{({\bf n})}\}_{\bf n}$ is the inverse of $\Gamma$.
By the strict upper triangularity of $\Gamma$ w.~r.~t.~to the coercive $\prec$,
we already
know that $\Gamma$ is invertible, so that it suffices to show $\tilde\Gamma \Gamma={\rm id}$,
which in turn follows from its transpose $\Gamma^* \tilde\Gamma^*={\rm id}$.
By the composition rule (\ref{pitilde})
established above,  $\Gamma^*\widetilde\Gamma^*$ is associated to
$\{\pi^{({\bf n})}+\Gamma^*\widetilde\pi^{({\bf n})}\}_{\bf n}$. 
By (\ref{as10}) we have that $\pi^{({\bf n})}+\Gamma^*\widetilde\pi^{({\bf n})}=0$, 
and learn from Lemma \ref{LEM:def} that ${\rm id}$ is associated with $0$.


We finally turn to the proof of \eqref{tri2}. We note that 
$\beta_1+\cdots+\beta_l=\beta$ implies the componentwise $\beta_j\le\beta$,
which by (\ref{as05}) implies $|\beta_j|_{\prec}\le|\beta|_{\prec}$.
Since every $\gamma$ with $[\gamma]\ge 0$ can be
written as the sum of $\gamma$'s of the form
\begin{align}\label{as08}
\gamma=e_{k}+e_{{\bf n}_1}+\cdots+e_{{\bf n}_j}\quad
\mbox{with}\quad j\le k,
\end{align}
we learn from (\ref{decomp1}) that we may assume that $\gamma$ is of this form.
Once more by (\ref{decomp1}) we have for these $\gamma$'s
\begin{align*}
(\Gamma^*)_\beta^\gamma=\sum_{\beta_0+\cdots+\beta_j=\beta}
(\Gamma^*)_{\beta_0}^{e_k}(\Gamma^*)_{\beta_1}^{e_{{\bf n}_1}}
\cdots(\Gamma^*)_{\beta_j}^{e_{{\bf n}_j}}.
\end{align*}
From (\ref{as02}) \& (\ref{as03}) we learn that this $(\Gamma^*)_\beta^\gamma$
is a linear combination of
\begin{align}\label{as11}
\pi^{(\0)}_{\beta_1'} \cdots \pi^{(\0)}_{\beta_{l}'} 
(\z_{\n_1}+\pi^{(\n_1)}_{})_{\beta_{1}} 
\cdots (\z_{\n_j}+\pi^{(\n_j)}_{})_{\beta_{j}},
\end{align}
where the multi-indices satisfy
\begin{align}\label{rg32}
\beta=e_{k+l}+\beta_1'+\cdots+\beta_l'+\beta_1+\cdots+\beta_j.
\end{align}
We need to show that the product (\ref{as11})
contains only factors $\pi_{\beta'}^{({\bf n})}$ with $\beta'\prec\beta$;
w.~l.~o.~g.~we may assume $l+j\geq1$.
To this purpose we apply $|\cdot|_{\prec}$ to \eqref{rg32};
by (\ref{as05}) and $|e_{k+l}|_{\prec}\ge|e_{k+l}|=\alpha(1+k+l)$ this implies
$|\beta|_{\prec}\ge\alpha(1+k-j)+|\beta_1'|_{\prec}+\cdots+|\beta_l'|_{\prec}
+|\beta_1|_{\prec}+\cdots+|\beta_j|_{\prec}$, which by $j\le k$ implies the
desired $|\beta_1'|_{\prec},\dots,|\beta_l'|_{\prec},
|\beta_1|_{\prec},\dots,|\beta_j|_{\prec}<|\beta|_{\prec}$.
\end{proof}


Finally, we show that the group $\mathsf{G}$ is large enough to contain 
the re-expansion maps.

\begin{lemma}\label{choice}
There exists $\{\pi_{yx}^{({\bf n})}\}_{\bf n}$ satisfying \eqref{def1}
such that the $\Gamma_{yx}\in \mathsf{G}$ associated by Lemma \ref{LEM:def} satisfies \eqref{cw27}.
\end{lemma}

As a consequence of working with a larger group than in \cite{LOT}, 
see Remark \ref{rem:1}, we don't have uniqueness of $\{\pi_{yx}^{({\bf n})}\}_{\bf n}$
and thus of $\Gamma_{yx}$. 
We refer the reader to \cite{T} for a uniqueness result 
when working with the smaller group.
An inspection of our construction reveals transitivity 
in line with \cite[Definition 3.3]{Hai16}
\begin{equation*}
\Gamma^*_{xy}\Gamma^*_{yz}=\Gamma^*_{xz}
\quad\textnormal{and}\quad
\Gamma^*_{xx}={\rm id},
\end{equation*}
see \cite[Section 5.3]{LOTT} for the argument;
it would also be a consequence of uniqueness. 

\begin{proof}[Proof of Lemma \ref{choice}]
We start by specifying $\pi_{yx\beta}^{({\bf n})}$ in the special cases of
${\bf n}={\bf 0}$ and of $\beta=e_{\bf m}$ for some ${\bf m}\not={\bf 0}$:
\begin{align}
\pi_{yx}^{({\bf 0})}&:=\Pi_y(x),\label{as12}\\
\pi_{yxe_{\bf m}}^{({\bf n})}&:=\left\{\begin{array}{cl}
\tbinom{{\bf m}}{\bf n} (x-y)^{{\bf m}-{\bf n}}&
\mbox{provided }{\bf n}<{\bf m}, \\
0 & \textnormal{otherwise}
\end{array}\right\}\quad\mbox{for}\;{\bf n}\not={\bf 0},\label{as13}
\end{align}
where ${\bf n}<{\bf m}$ means component-wise (non-strict) ordering 
and ${\bf n}\not={\bf m}$.
We note that (\ref{as12}) is necessary in order to bring the second
item of (\ref{cw27}) into agreement with the form (\ref{zk}).
We also remark that (\ref{as13}) yields by (\ref{zn})
\begin{align*}
(\Gamma_{yx}^*)_{e_{\bf m}}^{e_{\bf n}}
=\left\{\begin{array}{cl}
\tbinom{{\bf m}}{{\bf n}}(x-y)^{{\bf m}-{\bf n}}
&\mbox{provided}\;{\bf n}\le{\bf m},\\
0&\mbox{otherwise}
\end{array}\right\}.
\end{align*}
By the second part of (\ref{ks06}), which implies $(\Gamma_{yx}^*)_0^\gamma=0$ 
unless $\gamma=0$, by (\ref{as02}) in form of $(\Gamma_{yx}^*)_{e_{\bf m}}^{e_k}=0$,
and via (\ref{decomp1}) this strengthens to
\begin{align}\label{as26}
(\Gamma_{yx}^*)_{e_{\bf m}}^{\gamma}
=\left\{\begin{array}{cl}
\tbinom{{\bf m}}{{\bf n}}(x-y)^{{\bf m}-{\bf n}}
&\mbox{if}\;\gamma=e_{\bf n}\;\mbox{with}\;{\bf n}\le{\bf m},\\
0&\mbox{otherwise}
\end{array}\right\}.
\end{align}
The latter is imposed upon us by taking the
($\beta=e_{\bf m}$)-component of the first item in (\ref{cw27}) and plugging
in (\ref{ao59}). The second part of (\ref{as26}) implies that $\Gamma_{yx}$
maps the linear span of $\{\mathsf{z}_{\bf m}\}_{{\bf m}\not={\bf 0}}$ into itself;
since this linear span can be identified with the space $\mathbb{R}[x_1,x_2]/\mathbb{R}$
of space-time polynomials (modulo constants), this can be assimilated
to Hairer's postulate \cite[Assumption 3.20]{Hai16}.
We note that (\ref{as12}) and (\ref{as13}) satisfy (\ref{def1})
because of $|\cdot|\ge\alpha>0$, cf.~(\ref{as05}), 
and $|e_{\bf m}|=|{\bf m}|>|{\bf n}|$, respectively.
In line with (\ref{cw03}) and \cite{LOT}, we also set
\begin{align}\label{as25}
\pi_{yx\beta}^{({\bf n})}=0\quad\mbox{unless}\quad[\beta]\ge 0\;\;\mbox{or}\;\;
\beta=e_{\bf m}\;\mbox{for some}\;{\bf m}\not={\bf 0}.
\end{align}
%


It thus remains to construct $\pi_{yx\beta}^{({\bf n})}$ for ${\bf n}\not={\bf 0}$
and $[\beta]\ge 0$, which we will do by induction in $\beta$ w.~r.~t.~$\prec$.
According to (\ref{tri2}), we may consider $(\Gamma^*)_\beta^\gamma$ as
already constructed for $[\gamma]\ge 0$. According to (\ref{ks05}) and by the induction
hypothesis (\ref{cw27}), an inspection of the argument that leads from there to (\ref{cw28})
shows that we also have
\begin{align}\label{as14}
\Pi_{y\beta}^{-}=(\Gamma^*_{yx}\Pi_x^{-})_\beta.
\end{align}
The induction step consists in choosing 
$\{\pi_{yx\beta}^{({\bf n})}\}_{0<|{\bf n}|<|\beta|}$ such that
\begin{align}\label{as15}
\Pi_{y\beta}=(\Gamma^*_{yx}\Pi_x)_\beta+\Pi_{y\beta}(x)\overset{\eqref{as12}}{=}
(\Gamma^*_{yx}\Pi_x)_\beta+\pi_{yx\beta}^{({\bf 0})}.
\end{align}
Denoting by $P$ the projection on multi-indices $\gamma$ with $[\gamma]\geq0$,
so that by (\ref{ao59}) and (\ref{cw03}) we have 
$({\rm id}-P)\Pi_x$ $=\sum_{{\bf n}\not={\bf 0}}(\cdot-x)^{\bf n}
\mathsf{z}_{\bf n}$ and thus by (\ref{def1}) and (\ref{zn})
\begin{align}\label{as18}
(\Gamma_{yx}^*(1-P)\Pi_x)_\beta=\sum_{0<|{\bf n}|<|\beta|}(\cdot-x)^{\bf n}
\pi_{{yx}\beta}^{({\bf n})},
\end{align}
allows us to make $\{\pi_{yx\beta}^{({\bf n})}\}_{0<|{\bf n}|<|\beta|}$ 
in (\ref{as15}) explicit:
\begin{align}\label{cbp1}
(\Pi_y-\Gamma^*_{yx} P \Pi_x)_\beta
=\sum_{{\bf n}:|{\bf n}|<|\beta|}\pi^{({\bf n})}_{yx\beta} (\cdot-x)^{\bf n}. 
\end{align}
Hence our task reads
\begin{align}\label{as19}
(\Pi_y-\Gamma^*_{yx} P \Pi_x)_\beta=\mbox{polynomial of degree}\;<|\beta|. 
\end{align}
According to the PDE (\ref{cw04}), to (\ref{as14}), and to (\ref{as18}) we have
\begin{align}\label{as20}
(\partial_2-\partial_1^2)(\Pi_y-\Gamma^*_{yx} P \Pi_x)_\beta
=\mbox{polynomial of degree}\;<|\beta|-2.
\end{align}


In order to pass from (\ref{as20}) to (\ref{as19}),
we will now appeal to the uniqueness/Liouville statement in Lemma \ref{lem:int}
with $\eta=|\beta|$, which is $\not\in\mathbb{Z}$ according to (\ref{as17})
and $\ge\alpha$ according to (\ref{as05}), and $p=1$ for simplicity. 
More precisely, we apply Lemma \ref{lem:int} to
\begin{align*}
u=(\Pi_y-\Gamma^*_{yx} P \Pi_x)_\beta-\mbox{its Taylor polynomial in $x$ of order $<|\beta|$},
\end{align*} 
which makes sense since (\ref{as20}) implies that $(\Pi_y-\Gamma^*_{yx} P \Pi_x)_\beta$ is
smooth, and to $f\equiv 0$.
Hence for the assumption (\ref{ao55}) we need to check that
\begin{equation}\label{ks01}
\limsup_{z:|z-x|\uparrow\infty}\frac{1}{|z-x|^{|\beta|}}
\mathbb{E}|(\Pi_y-\Gamma^*_{yx}P\Pi_x)_\beta(z)|<\infty,
\end{equation}
which forces us to now become semi-quantitative.


By the estimate (\ref{cw01}) on $\Pi$, for (\ref{ks01}) it remains to 
show\footnote{which coincides with Hairer's postulate \cite[(3.2) in Definition 3.3]{Hai16}}
\begin{align}\label{ks93}
\mathbb{E}^{\frac{1}{p}}
|(\Gamma^*_{yx})_{\beta}^\gamma|^p\lesssim_{\beta,\gamma,p}|y-x|^{|\beta|-|\gamma|}
\quad\mbox{provided}\quad[\gamma]\ge 0.
\end{align}
In line with the language of \cite{LOTT}, we split the
argument for (\ref{ks93}) into an ``algebraic argument'', where we derive (\ref{ks93}) from
\begin{equation}\label{mt94}
\mathbb{E}^\frac{1}{p}|\pi_{yx\beta'}^{({\bf n})}|^p 
\lesssim_{\beta',p}|x-y|^{|\beta'|-|{\bf n}|}\quad\mbox{for}\;\beta'\prec\beta,
\end{equation}
and a ``three-point argument'', 
where we derive (\ref{mt94}) from the estimate (\ref{cw01}) on $\Pi$.


Here comes the argument for \eqref{ks93}, which is modelled after the one for (\ref{tri2})
in Lemma \ref{lem:group}.
By H\"older's inequality in probability and the additivity of $|\cdot|-\alpha$, cf.~(\ref{as05}), 
we may restrict to $\gamma$'s of the form (\ref{as08}). 
We are thus lead to estimate the product (\ref{as11}), 
which now takes the form of
\begin{align}\label{rg30}
\pi^{(\0)}_{yx\beta_1'} \cdots \pi^{(\0)}_{yx\beta_{l}'}
(\z_{\n_1}+\pi_{yx}^{(\n_1)})_{\beta_{1}}
\cdots (\z_{\n_j}+\pi_{yx}^{(\n_j)})_{\beta_{j}}.
\end{align}
Once again by H\"older's inequality, we infer from (\ref{mt94}) that the 
$\mathbb{E}^\frac{1}{p}|\cdot|^p$-norm of (\ref{rg30}) is
\begin{align*}
\lesssim |y-x|^{|\beta_1'|}\cdots|y-x|^{|\beta_{l}'|}
|y-x|^{|\beta_1|-|{\bf n}_1|}\cdots|y-x|^{|\beta_{j}|-|{\bf n}_j|}.
\end{align*}
By the additivity of $|\cdot|-\alpha$, the total exponent of $|y-x|$ can be 
identified with the desired expression:
\begin{align*}
\lefteqn{|\beta_1'|+\cdots+|\beta_l'|+(|\beta_1|-|{\bf n}_1|)+\cdots
+(|\beta_{j}|-|{\bf n}_j|)}\nonumber\\
&\overset{\eqref{rg32}}{=}|\beta|-|e_{k+l}|+(l+j)\alpha-(|{\bf n}_1|+\cdots+|{\bf n}_j|)
\overset{\eqref{as08}}{=}|\beta|-|\gamma|.
\end{align*}


Finally, we give the ``three-point argument'' for the estimate \eqref{mt94}, 
for notational simplicity in case of the current multi-index $\beta$,
so that we now may use \eqref{cbp1} and \eqref{ks93}. 
By \eqref{cw01} and \eqref{ks93}, the left hand side of \eqref{cbp1} can be estimated 
as follows
\begin{equation*}
\mathbb{E}^\frac{1}{p}|(\Pi_y-\Gamma^*_{yx} P \Pi_x)_\beta(z) |^p 
\lesssim_{\beta,p} (|z-x| + |y-x|)^{|\beta|}.
\end{equation*}
By the equivalence of norms on the finite-dimensional space of
space-time polynomials of degree $<|\beta|$, which by a duality argument can
be upgraded to the following estimate of annealed norms for random polynomials
\begin{equation*}
\max_{\n:\,|\n|<|\beta|} |y-x|^{|\n|}\,\mathbb{E}^\frac{1}{p}|\pi_{yx\beta}^{({\bf n})}|^p 
\lesssim
\fint_{|z-x|\le|y-x|}dz\,
\mathbb{E}^\frac{1}{p} \big| \sum_{\n:\,|\n|<|\beta|}(z-x)^\n \pi_{yx\beta}^{({\bf n})}\big|^p,
\end{equation*}
we obtain \eqref{mt94}.
\end{proof}

\medskip
\begin{center}
\sc Acknowledgements
\end{center}
\noindent
This lecture note is based on a course given at the Institute of Science and Technology Austria (ISTA) in July 2022.
The authors would like to express their gratitude to ISTA for its hospitality.	



\bigskip

\bigskip

\begin{flushleft}
\footnotesize \normalfont
\textsc{Felix Otto, Kihoon Seong, and Markus Tempelmayr\\
Max--Planck Institute for Mathematics in the Sciences\\ 
04103 Leipzig, Germany}\\
\texttt{\textbf{felix.otto@mis.mpg.de, kihoonseong@cornell.edu, markus.tempelmayr@uni-muenster.de}}
\end{flushleft}

\end{document}